\documentclass{article}
\usepackage[margin=1in]{geometry}

\usepackage[utf8]{inputenc}

\usepackage{braket,amsfonts}
\usepackage{cite}
\usepackage{amsfonts}
\usepackage{amsmath,amscd,amsthm,amssymb,mathrsfs,setspace}
\usepackage{mathtools}
\usepackage{amssymb}
\usepackage{algorithmicx}
\usepackage{algpseudocode}
\usepackage{xcolor}
\usepackage{url}
\usepackage[colorlinks]{hyperref}
\usepackage[nameinlink,capitalize]{cleveref}
\usepackage{pgfplots}
\usepackage{tikz}
\usepackage{doi}
\usepackage{subcaption}
\usepackage{siunitx}
\usepackage[nolist]{acronym}
\usepackage{bbm}
\usepackage{algorithm}
\usepackage{algorithmicx}
\usepackage{fancyhdr}
\usepackage{todonotes}
\usepackage{booktabs}
\usepackage{enumitem}
\usepgfplotslibrary{external}

\newtheorem{theorem}{Theorem}[section]
\newtheorem{proposition}[theorem]{Proposition}
\newtheorem{lemma}[theorem]{Lemma}
\newtheorem{corollary}[theorem]{Corollary}

\theoremstyle{definition}
\newtheorem{definition}[theorem]{Definition}
\newtheorem{assumption}[theorem]{Assumption}
\newtheorem{example}[theorem]{Example}

\theoremstyle{remark}
\newtheorem{remark}[theorem]{Remark}

\DeclareMathOperator*{\supp}{supp}

\DeclareMathOperator*{\TV}{TV}
\DeclareMathOperator*{\BV}{BV}
\DeclareMathOperator*{\BVV}{BV_V}
\DeclareMathOperator*{\dvg}{div}
\DeclareMathOperator*{\pred}{pred}
\DeclareMathOperator*{\ared}{ared}

\newcommand*\dd{\mathop{}\!\mathrm{d}}

\newcommand{\weakstarto}{\stackrel{\ast}{\rightharpoonup}}
\newcommand{\weakto}{\rightharpoonup}

\newcommand{\bdvg}[1]{\dvg{}_{#1}}

\newcommand{\mres}{\mathbin{\vrule height 1.6ex depth 0pt width 0.13ex\vrule height 0.13ex depth 0pt width 1.3ex}}

\newcommand{\R}{\mathbb{R}}
\newcommand{\N}{\mathbb{N}}
\newcommand{\Z}{\mathbb{Z}}
\newcommand{\Ha}{\mathcal{H}}

\title{On Integer Optimal Control
with Total Variation Regularization
on Multi-dimensional Domains}

\author{Paul Manns\thanks{Faculty of Mathematics,
		TU Dortmund University (\url{paul.manns@tu-dortmund.de}).}
	\and 
	Annika Schiemann\footnotemark[1] \thanks{Previous versions of this article were published under Annika Schiemann's former name Annika M\"uller.}}

\begin{document}
	\maketitle
	\begin{abstract}
		We consider optimal control problems with integer-valued controls and a
		total variation regularization penalty in the objective
		on domains of dimension
		two or higher. The penalty yields that the feasible set is sequentially closed in the weak-$^*$
		and closed in the strict topology in the space of functions of bounded variation.
		
		In turn, we derive first-order optimality conditions of the optimal control
		problem as well as trust-region subproblems with partially linearized model
		functions using local variations of the level sets of the
		feasible control functions. We also prove that a recently proposed function space
		trust-region algorithm---sequential linear integer programming---produces
		sequences of iterates  whose limits are first-order optimal points.
	\end{abstract}

	
	\section{Introduction}
	
	Let $\alpha > 0$. Let $2 \le d \in \N$.
	Let $\Omega \subset \R^d$ be a bounded Lipschitz domain.
	We are concerned with 
	optimization problems of the form
	\begin{gather}\label{eq:p}
	\begin{aligned}
	\min_{v \in L^2(\Omega)}\ & J(v) \coloneqq F(v) + \alpha \TV(v) \\
	\text{s.t.}\ & v(x) \in V \coloneqq \{\nu_1,\ldots,\nu_M\} \subset \Z
	\text{ for almost all (a.a.)\ } x\in \Omega,
	\end{aligned}\tag{P}
	\end{gather}
	where $F : L^2(\Omega) \to \R$ is lower semicontinuous and 
	bounded below and $\TV : L^1(\Omega) \to [0,\infty]$ is the
	total variation seminorm.
	Specifically, we are interested in
	\emph{integer optimal control problems}, in which
	$F$ has the form $F = j \circ S$ for an objective function
	$j : Y \to \R$ and the (potentially non-linear) solution operator $S : L^2(\Omega) \to Y$
	of a PDE (or $S$ is another kind of integral operator) with state space $Y$.
	The distinctive feature of \eqref{eq:p} is the integrality constraint on the control $v$.
	
	Integer optimal control problems are a versatile mathematical problem class that
	allows for modeling many scenarios with real world applications, for example,
	transmission line control \cite{gottlich2019partial}, traffic light control 
	\cite{goettlich2017partial}, gas network control 
	\cite{martin2006mixed,hante2017challenges},
	and automotive control \cite{gerdts2005solving}.
	Owing to these applications, integer optimal control has attracted
	significant interest in recent years. We highlight
	the research on the \emph{combinatorial integral approximation}
	\cite{sager2011combinatorial,hante2013relaxation,sager2012integer,manns2020multidimensional,kirches2020compactness,hante2021time}, which is closely related to relaxation-based approaches in shape and topology optimization \cite{allaire2001shape,sigmund2013topology}
	and compensated compactness \cite{tartar1979compensated}.
	Recently, descent optimization algorithms that produce integer-valued controls directly 
	without solving the relaxation have also been investigated,
	see, for example, \cite{hahn2022binary,vogt2022mixed}. They are based
	on the same principles as the \emph{combinatorial integral approximation},
	which is shown in \cite{manns2022on}.
	The aforementioned results can provide optimal approximation
	guarantees on the computed integer-valued control functions,
	for the case $\alpha = 0$, where no regularization
	is present. In \cite{manns2021relaxed} it has been shown
	that the convex relaxation of a so-called multi-bang regularizer
	\cite{clason2014multi,clason2018vector} may
	be integrated into the \emph{combinatorial integral approximation}.
	
	However, without enforcing more regularity on the control input
	$v$ in \eqref{eq:p} these approaches lead to
	approximations of functions with convex codomains in weak topologies of
	$L^p$ spaces,
	see also the Lyapunov convexity theorem 
	\cite{lyapunov1940completely,lindenstrauss1966short}, which means that
	close approximations may exhibit highly oscillating functions with
	their total variation tending towards infinity 
	\cite{hante2013relaxation}. While the approximation algorithms that
	compute the integer controls have been tailored to reduce oscillations
	\cite{bestehorn2021mixed,sager2021mixed}
	this effect is inevitable.
	
	Therefore, it has been proposed to enforce control functions of bounded
	variation by choosing $\alpha > 0$ in
	\cite{leyffer2021sequential}.
	Driven by applications in mathematical 
	imaging, total variation regularization has been analyzed in depth, see
	\cite{rudin1992nonlinear,vogel1996iterative,dobson1996analysis,chambolle1997image,chambolle2010introduction,lellmann2014imaging,hintermuller2017optimal},
	and also entered research on optimal control problems, see
	\cite{loxton2012control,kaya2020optimal,engel2021optimal,hafemeyer2022path}.
	Combining $\alpha > 0$ with the integrality constraint $v(x) \in V$ a.e.,
	it follows that the feasible set of \eqref{eq:p} is sequentially weak-$^*$ closed
	in the space of functions of bounded variation.
	Thus \eqref{eq:p} admits a solution in this setting, see, e.g., the analysis in
	\cite[Chap.\ 4-5]{ambrosio2000functions}, \cite[Cor.\ 2.6]{burger2012exact}, or
	\cite[Prop.\ 2.3]{leyffer2021sequential}.
	
	Because of the sequential weak-$^*$ closedness of the feasible set,
	a novel trust-region algorithm
	that operates on the feasible set is proposed in  \cite{leyffer2021sequential}. 
	The model function of the trust-region subproblems
	is the sum of a linearization of $F$ and the term $\alpha \TV$, meaning
	that the latter is kept exactly. The trust region is an $L^1$-ball around the
	current iterate. After discretization, this leads to linear integer programs.
	Optimality conditions for \eqref{eq:p} and asymptotics of the trust-region 
	algorithm have been analyzed for the case $d = 1$, that is if $\Omega$ is
	an open interval. For this case, the trust-region subproblems can be solved
	efficiently with the strategies proposed in 
	\cite{severitt2022efficient,marko2022integer} and it has been noted 
	in \cite{leyffer2021sequential} that when the sequence of heights of the
	steps of the control function settles during the optimization procedure, 
	optimizing \eqref{eq:p} becomes a
	switching point optimization, for which optimization techniques
	have been investigated
	\cite{ruffler2016optimal,de2020sparse,maurer2004second,stellato2017second}.
	A proximal objective instead of a trust-region globalization
	has been proposed in \cite{marko2022integer}.
	
	None of these works has addressed the case that the domain is multi-dimensional, that is $d \ge 2$, where the geometric properties of the $\TV$-term differ
	fundamentally from $d = 1$, however. This work closes this gap.
	
	\paragraph{Contribution} We analyze optimality conditions
	for \eqref{eq:p} and the function space algorithm from 
	\cite{leyffer2021sequential} for the case $d \ge 2$.
	In particular, we employ and extend results on geometric
	variational problems to derive first-order optimality conditions for
	\eqref{eq:p} and trust-region subproblems using so-called \emph{local 
		variations}. The underlying variational principle gives rise to a
	sufficient decrease condition
	(related to Cauchy point computations in nonlinear programming).
	We leverage this insight to prove convergence of the trust-region
	algorithm to first-order optimal points. 
	
	As an intermediate result, we verify $\Gamma$-convergence, more specifically
	the $\liminf$- and $\limsup$-inequalities for the trust-region subproblems
	with respect to convergence of the (partial) linearization point for the
	subproblem. Notably, this result is not true for $d = 1$ but opens
	possibilities for algorithmic improvements, in particular, sources of
	inexactness in the trust-region subproblems.
	
	While rigorous numerical analysis and computational assessment
	are not our focus and beyond the scope of this article, we
	nevertheless provide first computational experiments in
	\cref{sec:comp} to allow some impression on the algorithm in practice.
	
	\paragraph{Structure of the remainder} We introduce some notation and functions
	of bounded variation in \cref{sec:bv} followed by local variations and 
	preparations in \cref{sec:local_variations}.
	We define locally optimal solutions and prove a first-order optimality condition for \eqref{eq:p} in \cref{sec:sol}.
	We define and analyze the trust-region subproblems and introduce the 
	algorithm in \cref{sec:algorithm}.
	Then we derive our variational stationarity concept and prove the convergence
	of the iterates of the trust-region algorithm to stationary points in
	\cref{sec:trm_analysis_mg2}. We provide computational experiments
	in \cref{sec:comp} and concluding remarks in
	\cref{sec:conclusion}.
	
	\section{Notation and Primer on Functions of Bounded Variation}\label{sec:bv}
	The symmetric difference of the sets $A$, $B \subset \R^d$ is $A\Delta B$.
	The Lebesgue  measure is denoted by $\lambda$. The restriction of
		a measure $\mu$ to a set $A$ is denoted by $\mu \mres A$.
	For a Lebesgue space $L^p(\Omega)$ on 
	$\Omega$, we abbreviate $\|\cdot\|_{L^p(\Omega)}$ by $\|\cdot\|_{L^p}$.
	We abbreviate the $L^2(\Omega)$-inner 
	product by $(\cdot,\cdot)_{L^2}$. The $\{0,1\}$-valued indicator function
	of a set $A$ is $\chi_A$.
	Let $E \subset \Omega$ be measurable. We define \emph{the perimeter of $E$ in $\Omega$}, see, e.g,
	\cite[Def.\ 3.35]{ambrosio2000functions}, as
	\[ P(E,\Omega) \coloneqq \sup \left\{\int_E \dvg \varphi(x)\dd x \,\Big|\, \varphi \in (C_c^1(\Omega))^d \text{ and }
	\sup_{x \in \Omega}\|\varphi(x)\| \le 1 \right\}.
	\]
	If $P(E,\Omega) < \infty$, $E$ is called a Caccioppoli set. A partition
	$\{E_i\}_{i \in I}$ of $\Omega$ is called a Caccioppoli partition if
	$\sum_{i \in I} P(E_i,\Omega) < \infty$. We denote the topological boundary of $E$
	by $\partial E$ and the reduced boundary of $E$ by $\partial^* E$, see \cite[Def.\ 3.54]{ambrosio2000functions}.
	We denote the essential boundary,
	the set of points that have neither density $1$ nor $0$ with respect to $A$,
	see \cite[Def.\ 3.54]{ambrosio2000functions},
	by $\partial^e A$.
	Note that $\partial^* E$ may intersect with $\partial^* \Omega$
	so that $P(E,\Omega) = \Ha^{d-1}(\partial^* E \cap \Omega)$
	while $P(E,\Omega) \neq \Ha^{d-1}(\partial^* E)$.
	We recall that a function $u \in L^1(\Omega)$ is of bounded variation (in $\BV(\Omega)$) if its distributional
	derivative $Du$ is a finite Radon measure in $\Omega$ \cite[Chap.\ 3]{ambrosio2000functions}.
	In other words $\TV(u) \coloneqq |Du|(\Omega) < \infty$, where $|\mu|$ denotes the total variation measure of a measure $\mu$.
	Most notably $\TV(\chi_E) = P(E,\Omega)$ for a Borel set $E \subset \Omega$.
	A sequence of functions $(v^n)_n \subset \BV(\Omega)$ converges
	\emph{weakly-$^*$} to $v \in \BV(v)$, denoted by $v^n \weakstarto v$, if $v^n \to v$ in $L^1(\Omega)$ and $\limsup_{n \to \infty} \TV(v^n) < \infty$.
	Moreover, $(v^n)_n$ converges \emph{strictly} to $v \in \BV(v)$
	if $v^n \to v$ in $L^1(\Omega)$ and $\TV(v^n) \to \TV(v) < \infty$.
	
	Feasible points for \eqref{eq:p} are functions in $\BV(\Omega)$
	that attain values only in the finite set $V$. Their distributional derivatives are absolutely continuous with respect to $\Ha^{d-1}$ \cite[Theorems 3.36 and 3.59]{ambrosio2000functions}. In particular, they are also so-called $SBV$-functions ($S$ for Special), see \cite[Sec.\ 4.1]{ambrosio2000functions}. We define
	\[ \BVV(\Omega) \coloneqq \left\{ v \in \BV(\Omega) \,|\, v(x) \in V \text{ for a.a.\ } x \in \Omega \right\} \subset \BV(\Omega), \]
	which is the feasible set of \eqref{eq:p}.
	We will use the following lemma (proven in \cref{sec:auxilary}).
	\begin{lemma}\label{lem:TV_of_BVV}
		\begin{enumerate}[label=\emph{(\alph*)}]
			\item $\BVV(\Omega)$ is sequentially weakly-$^*$ and strictly closed in $\BV(\Omega)$.
			\item Let $v \in \BVV(\Omega)$. Then there exists a Caccioppoli partition $\{E_1,\ldots,E_M\}$ of $\Omega$
			such that $v = \sum_{i=1}^{M} \nu_i \chi_{E_i}$. \label{lem:TV_of_BVV_b}
			\item Let $\sum_{i=1}^{M} \nu_i \chi_{E_i} = v \in \BV_V(\Omega)$ as in \ref{lem:TV_of_BVV_b}.
			Then it holds that
			\begin{align}\label{eq:tv_identity}
			\infty > \TV(v) &=  | Dv | (\Omega) =
			\sum_{i=1}^{M-1} \sum_{j=i+1}^{M} | \nu_i - \nu_j | \mathcal{H}^{d-1} (\partial^* E_i \cap \partial^* E_j)  \text{ and}\\
			\label{eq:Ek_has_finite_perimeter}
			\TV(v) &\geq \frac{1}{2} \sum_{i =1}^{M} P(E_i,\Omega).			
			\end{align}
		\end{enumerate}
	\end{lemma}
	
	\section{Local Variations}\label{sec:local_variations}
	A key concept to derive sensible notions of stationarity for 
	\eqref{eq:p} and subsequently a sufficient decrease condition
	for our algorithmic framework are \emph{local variations},
	which allow to analyze smooth perturbations of the boundaries
	of the level sets of feasible control functions.
	We introduce the relevant concepts leaning on \cite{maggi2012sets}.
	\begin{definition}
		\begin{enumerate}[label=\emph{(\alph*)}]
			\item A \emph{one parameter family of diffeomorphisms of $\R^d$} is
			a smooth function $f:  (-\varepsilon,\varepsilon)  \times \R^d\to\R^d$
			for some $\varepsilon > 0$ such that for all $t \in (-\varepsilon,\varepsilon)$,
			the function $f_t(\cdot) \coloneqq f(t,\cdot) : \R^d \to\R^d$ is a diffeomorphism.
			\item Let $A \subset \R^d$ be open. Then the family
			$(f_t)_{t\in(-\varepsilon,\varepsilon)}$ is a
			\emph{local variation in $A$} if additionally to (a) we have
			$f_0(x) = x$ for all $x \in \R^d$ and there is a compact set $K \subset A$
			such that $\{x \in \R^d\,|\,f_t(x) \neq x\} \subset K$
			for all  $t \in (-\varepsilon,\varepsilon)$.
			\item For a local variation, we define its \emph{initial velocity}:
			$\phi(x) \coloneqq \frac{\partial f}{\partial t}(0, x)$ for $x \in \R^d$.
		\end{enumerate}
	\end{definition}
	We recall basic properties and a certain local variation that is
	important in the remainder.
	\begin{proposition}\label{prp:elementary_local_variation_properties}
		Let $\phi \in C_c^\infty(\Omega,\R^d)$. Let
		$f_t \coloneqq I + t\phi$ for $t \in \R$.
		Then $(f_t)_{t\in(-\varepsilon,\varepsilon)}$
		is a local variation in $\Omega$ with initial velocity $\phi$
		for some $\varepsilon > 0$.
	\end{proposition}
	\begin{proof}
		Because of $\| \nabla f_t(x) - I \| = \|t \nabla \phi(x) \| = |t| \| \nabla \phi(x) \| \leq |t| M$
		with $M \coloneqq \max_{x \in \bar{\Omega}} \| \nabla \phi(x) \| < \infty$, \cref{lem:derivative_arg} yields the existence of some $\varepsilon >0$ such that $(f_t)_{t \in (-\varepsilon,\varepsilon)}$ is a one parameter family of diffeomorphisms.
		We have $f_0(x) = x$ for all $x \in \R^d$ and $\{ x \in \R^d | f_t(x) \neq x \} \subset \supp \phi \subset \Omega$ for all $t \in \R$. Hence, $(f_t)_{t \in (-\varepsilon,\varepsilon)}$ is a local variation in $\Omega$ with initial velocity $\frac{\partial f}{\partial t} (0,x) = \phi(x)$ for $x \in \R^d$.
	\end{proof}
	Let $v = \sum_{i=1}^M \nu_i \chi_{E_i}$ for a Caccioppoli partition $\{E_1, \ldots, E_M\}$ of $\Omega$.
	Let $(f_t)_{t\in(-\varepsilon,\varepsilon)}$ be a
	local variation in $\Omega$ with initial velocity $\phi$.
	Because the $f_t$ are diffeomorphisms it follows that
	$\{f_t(E_1),\ldots,f_t(E_M)\}$ is a partition of $\Omega$
	that induces the piecewise constant function
	\[ f_{t}^{\#}v \coloneqq \sum_{i=1}^M \nu_i \chi_{f_t(E_i)}. \]
	For a $C_c^\infty(\Omega;\R^d)$-function $\phi$ we
	introduce the notation $\bdvg{E} \phi: \partial^*E \to \R$
	for the so-called \emph{boundary divergence of $\phi$ on $E$},
	\[ \bdvg{E}\phi(x) \coloneqq \dvg \phi(x) - n_E(x) \cdot \nabla \phi(x)n_E(x), \]
	where $n_E$ denotes the unit outer normal vector on the
	reduced boundary of $E$. We obtain Taylor expansions
	for $(g, f^{\#}_t v)_{L^2}$ and $\TV(f^{\#}_t v)$ with respect to $t$,
	which we prove below.
	\begin{lemma}[Extension of Theorem 17.5 in \cite{maggi2012sets}]\label{lem:Taylor_TV_mg2}
		Let $\{E_1,\ldots,E_M\}$ be a Caccioppoli partition of $\Omega$,
		and $v = \sum_{i=1}^M \nu_i \chi_{E_i}$.
		Let $(f_t)_{t\in(-\varepsilon,\varepsilon)}$ be a
		local variation in $\Omega$ with initial velocity $\phi$.
		Then there exists $\varepsilon_0 > 0$ such that
		\[ \TV(f_t^{\#}v) =
		\TV(v) + t\sum_{i=1}^M \sum_{j=i + 1}^M|\nu_i - \nu_j|
		\int_{\partial^*E_i \cap \partial^* E_j}\bdvg{E_i} \phi(x)\dd \Ha^{d-1}(x)
		+ O(t^2)
		\]
		for all $t \in (-\varepsilon_0,\varepsilon_0)$.
	\end{lemma}
	\begin{proof}
		The claim follows with the identities
		\begin{align*}
		\hspace{2em}&\hspace{-2em}\TV(f_t^{\#}v) \\
		&= \sum_{i=1}^M \sum_{j=i + 1}^M|\nu_i - \nu_j| \Ha^{d - 1}(\partial^* f_t(E_i) \cap \partial^* f_t(E_j))\\
		&= \sum_{i=1}^M \sum_{j=i + 1}^M|\nu_i - \nu_j| \Ha^{d - 1}(\partial^* f_t(E_i) \cap f_t(\partial^* E_j))\\
		&= \sum_{i=1}^M \sum_{j=i + 1}^M|\nu_i - \nu_j| \int_{\partial^* E_i \cap \partial^* E_j}
		\mathcal{J}_xf_t(x) \|( (\nabla f^{-1}_t \circ f_t)(x))^T n_{E_i}(x)\|\dd \Ha^{d-1}(x)\\
		&= \sum_{i=1}^M \sum_{j=i + 1}^M|\nu_i - \nu_j|
		\int_{\partial^*{E_i} \cap \partial^* E_j}
		(1 + t\bdvg{E_i} \phi(x)) \dd \Ha^{d-1}(x)
		+ O(t^2) \\
		&= \TV(v) + t\sum_{i=1}^M \sum_{j=i + 1}^M|\nu_i - \nu_j|
		\int_{\partial^*E_i \cap \partial^* E_j}\bdvg{E_i} \phi(x)\dd \Ha^{d-1}(x)
		+ O(t^2)
		\end{align*}
		that are due to (in the order of their appearance):
		\eqref{eq:tv_identity}, Prop.\,17.1 (17.4) in \cite{maggi2012sets},
		Prop.\,17.1 (17.6) in \cite{maggi2012sets},
		the proof of Thm 17.5 in \cite{maggi2012sets},
		and \eqref{eq:tv_identity}. Note that $\mathcal{J}_x f_t$ denotes the Jacobian
			determinant of the function $f_t$.
	\end{proof}
	\begin{corollary}
		Let $\{E_1,\ldots,E_M\}$ be a Caccioppoli partition of $\Omega$,
		Let $(f_t)_{t\in(-\varepsilon,\varepsilon)}$ be a
		local variation with initial velocity $\phi \in C_c^\infty(\Omega;\R^d)$.
		Then there exists $\varepsilon_0 > 0$ such that
		the partition $\{f_t(E_1),\ldots,f_t(E_M)\}$
		is a Caccioppoli partition
		for all $t \in (-\varepsilon_0,\varepsilon_0)$.
	\end{corollary}
	\begin{proof}
		The claim follows from \cref{lem:Taylor_TV_mg2} and
		\eqref{eq:tv_identity}.
	\end{proof}
	
	\begin{lemma}[Extension of Proposition 17.8 in \cite{maggi2012sets}]\label{lem:Taylor_lin_mg2}
		Let $(f_t)_{t\in(-\varepsilon,\varepsilon)}$ be a
		local variation with initial velocity $\phi \in C_c^\infty(\Omega;\R^d)$.
		Let $g \in C(\bar{\Omega})$.
		Then it follows that
		\[ \int_{\Omega} g(x)(f_{t}^{\#}v(x) - v(x))\dd x
		=t \sum_{i=1}^M \nu_i
		\int_{\partial^*E_i \cap \Omega} g(x)\left(\phi(x)\cdot n_{E_i}(x)\right)
		\dd \Ha^{d-1}(x) + o(t).
		\]
	\end{lemma}
	\begin{proof}
		The claim follows from the following identities, where
		the second is due to Proposition 17.8 in
		Section 17.3 in \cite{maggi2012sets}:
		\begin{align*}
		\int_{\Omega} g(x)(f_{t}^{\#}v(x) - v(x))\,\dd x
		&= \sum_{i=1}^M \nu_i
		\int_\Omega g(x)(\chi_{f_t(E_i)}(x) - \chi_{E_i}(x))\,\dd x\\
		&= t \sum_{i=1}^M \nu_i
		\int_{\partial^*E_i \cap \Omega} g(x)\left(\phi(x)\cdot n_{E_i}(x)\right)
		\dd \Ha^{d-1}(x) + o(t).
		\end{align*}
	\end{proof}
	If a local variation $(f_t)_{t \in (-\varepsilon,\varepsilon)}$ is given as in \cref{prp:elementary_local_variation_properties}, we can obtain Lipschitz continuity for the measure of
	the intersection of a Caccioppoli set $F$ and a transformed Caccioppoli set $f_t(E)$.
	\begin{lemma}\label{lem:Lipschitz_symmetric_difference}
		Let $\phi \in C_c^\infty(\Omega; \R^d)$. Let $(f_t)_{t\in(-\varepsilon,\varepsilon)}$ be the local variation defined by $f_t \coloneqq I + t \phi$ for $t \in (-\varepsilon,\varepsilon)$.
		Then there exist $0 < \varepsilon_0 < \varepsilon$ and $L > 0$ such that for all $s$, $t \in (-\varepsilon_0,\varepsilon_0)$
		and all Caccioppoli sets $E$, $F$ in $\Omega$ it holds that
		\[ \left|\lambda(f_t(E)\cap F) - \lambda(f_s(E)\cap F)\right|
		\le L |t - s| P(E,\Omega). \]
	\end{lemma}
	\begin{proof}
		We adapt the strategy of the proof of Lemma 17.9 in \cite{maggi2012sets} that shows the bound $\lambda(f_{t}(E)\Delta E) \le L |t| \TV(\chi_E)$.
		Let $g_t \coloneqq f_t^{-1}$ for all $t \in (-\varepsilon,\varepsilon)$.
		Let $s$, $t \in (-\varepsilon,\varepsilon)$ be given.
		We define $u_\delta \in \BV (\Omega) \cap C^\infty(\Omega)$ as in \cite[Thm. 5.3]{evans2015measure} so that
		\begin{gather}\label{eq:partition_of_unity}
		u_\delta \to \chi_E\text{ in }L^1(\Omega)\enskip\text{and}\enskip
		\int_\Omega \|\nabla u_\delta(x)\|\dd x = | D u_\delta |(\Omega) \to |D \chi_E| (\Omega)
		\enskip\text{as}\enskip\delta \to 0.
		\end{gather}
		This implies 
		\begin{align}
		\left|\lambda(f_t(E)\cap F) - \lambda(f_s(E)\cap F)\right|
		&= \left|\int_{F} |\chi_E(g_{t}(x))|\dd x
		- \int_{F} |\chi_E(g_{s}(x))|\dd x\right| \nonumber \\
		&\le \int_{F} |\chi_E(g_{t}(x)) - \chi_E(g_{s}(x))|\dd x \label{eq:int} \\
		&=\lim_{\delta \to 0} \int_{F} |u_\delta(g_{t}(x)) - u_\delta(g_{s}(x))|\dd x, \nonumber
		\end{align}
		where we have used the reverse triangle inequality to obtain the inequality.
		Before we continue to analyze
		the right hand side of this estimate, we need some preparations.
		
		We define $G_{t,s,\alpha}(x) \coloneqq \alpha g_{t}(x) + (1 - \alpha)g_{s}(x)$
		for $\alpha \in [0,1]$ and $x \in \R^d$.
		For the derivative of $G_{t,s,\alpha}(x)$ with respect to $\alpha$, we obtain
		\[ \frac{\partial}{\partial\alpha} G_{t,s,\alpha}(x)
		=  g_{t}(x) - g_{s}(x).
		\]
		For all (arbitrary but fixed) $\alpha \in [0,1]$,
		we compute the Jacobian determinant of
		$G_{t,s,\alpha}(x)$ with respect to $x$ and obtain
		\[ \mathcal{J}_x G_{t,s,\alpha}(x) = \sqrt{\det\left(
			(\alpha \nabla g_t(x) + (1 - \alpha) \nabla g_s(x))
			(\alpha \nabla g_t(x) + (1 - \alpha) \nabla g_s(x))^T\right).}
		\]
		\cref{lem:g_t} gives $\nabla g_\tau(x) \to I$ for $\tau \to 0$ uniformly for $x \in \R^d$.
		Consequently, we obtain uniformly for all $x \in \R^d$ that
		\[ \|I -\nabla G_{s,t,\alpha}(x)\| 
		\le \alpha \|I - \nabla g_s(x)\| + (1 - \alpha) \|I - \nabla g_t(x)\| \eqqcolon \delta,
		\]
		where $\delta$ can be made arbitrarily small also
		uniformly for $\alpha \in [0,1]$ by only allowing sufficiently
		small absolute values for $s$ and $t$. Thus by virtue of \cref{lem:derivative_arg} there exists $\varepsilon_{0} > 0$
		such that the function $G_{s,t,\alpha}$ is invertible for all $s$, $t \in (-\varepsilon_0,\varepsilon_0)$ and 
		all $\alpha \in [0,1]$. Moreover, after possibly reducing 
		$\varepsilon_0$ further, we obtain for all
		$s$, $t \in (-\varepsilon_0,\varepsilon_0)$ and all $x \in \R^d$ that
		\begin{gather}\label{eq:lb_on_Jacobian_determinant}
		\mathcal{J}_x G_{t,s,\alpha}(x) \ge 0.5.
		\end{gather}
		
		We continue our analysis of $d_\delta \coloneqq
		\int_{F} |u_\delta(g_{t}(x)) - u_\delta(g_{s}(x))|\dd x$
		and deduce
		\begin{align*}
		d_\delta
		&= \int_{F} |u_\delta(G_{t,s,1}(x)) - u_\delta(G_{t,s,0}(x))|\dd x &&\text{\small{Definition of $G_{s,t,\alpha}(x)$}}\\
		&= \int_{F} \left|
		\int_0^1\frac{\partial}{\partial \alpha}u_\delta(G_{t,s,\alpha}(x))\Big|_{\alpha=\beta} \dd\beta\right|\dd x
		&&\text{\small{Fundamental theorem of calculus}}\\
		&\le \int_{F} 
		\int_0^1 \left\|\nabla u_\delta(G_{t,s,\alpha}(x))\right\|\dd\alpha
		\left\|g_{t}(x) - g_{s}(x)\right\|\dd x. &&\text{\small{Chain rule and submultiplicativity}}
		\end{align*}
		By \cref{lem:g_t} there exists $C_2 > 0$ such that
		$\left\|g_{t}(x) - g_{s}(x)\right\| \le C_2 |t - s|$.
		We insert this inequality and apply Fubini's theorem to obtain
		the estimate
		\begin{gather}\label{eq:ddelta_Lipschitz_bound}
		d_\delta \le C_2|t - s|
		\int_0^1 \int_{F} 
		\left\|\nabla u_\delta(G_{t,s,\alpha}(x))\right\|\dd x \dd\alpha.
		\end{gather}
		Next, we observe
		\begin{gather}\label{eq:invariant_set}
		G_{t,s,\alpha}(\Omega) = \Omega
		\end{gather}
		To see this, we note that if $x \notin \supp \phi$, then $x = G_{t,s,\alpha}(x)$.
		Combining this with $\supp \phi \subset \Omega$, we obtain
		$G_{s,t,\alpha}(\Omega) = \Omega$ because $G_{s,t,\alpha}$ is invertible
		for all $s$, $t \in (-\varepsilon_0,\varepsilon_0)$ and all $\alpha \in [0,1]$.
		We apply the change of variables formula and the inverse function theorem
		to obtain
		\begin{align}
		\int_{F} 
		\left\|\nabla u_\delta(G_{t,s,\alpha}(x))\right\|\dd x
		&= \int_{F} 
		\frac{\left\|\nabla u_\delta(G_{t,s,\alpha}(x))\right\|}{\mathcal{J}_x G_{t,s,\alpha}(x)}\mathcal{J}_x G_{t,s,\alpha}(x)\dd x 
		&&\text{\small{$\cdot 1$}}\notag\\ 
		&= \int_{G_{t,s,\alpha}(F)}
		\frac{\left\|\nabla u_\delta(y)\right\|}{\mathcal{J}_xG_{t,s,\alpha}(G_{t,s,\alpha}^{-1}(y))} \dd y
		&&\text{\small{Area formula}}\notag\\
		&  
		\le 2 \int_{\Omega}\left\|\nabla u_\delta(y)\right\|\dd y.
		&&\text{\small{\eqref{eq:lb_on_Jacobian_determinant}, \eqref{eq:invariant_set}}}
		\label{eq:udelta_uniform_gradl1_bound}
		\end{align}

		We insert the
		estimate \eqref{eq:udelta_uniform_gradl1_bound} into \eqref{eq:ddelta_Lipschitz_bound} and pass to
		the limit $\delta \to 0$, which gives
		\[ d_\delta \le 2 C_2\int_{\Omega}\|\nabla u_\delta(y)\|\dd y|t - s|
		\,\underset{\eqref{eq:partition_of_unity}}\to
			2 C_2 |D\chi_E|(\Omega)|t - s|
			=\, 2 C_2P(E,\Omega)|t - s|.
		\]
		The claim follows by combining these considerations with the
		estimate $|\lambda(f_t(E) \cap F) - \lambda(f_s(E) \cap F)|
		\le \lim_{\delta\to 0}d_\delta$ and the choice $L\coloneqq
		2 C_2$.
	\end{proof}
		\begin{remark}
			From the proof of \cref{lem:Lipschitz_symmetric_difference} it can also be derived that
			\begin{align*}
			\lambda(f_t(E) \Delta f_s(E)) \leq L |t-s| P(E,\Omega)
			\end{align*}
			by using that $\lambda(f_t(E) \Delta f_s(E)) = \int_\Omega | \chi_E(g_t(x)) - \chi_E(g_s(x)) | \dd x$
			and continuing the proof from \eqref{eq:int} with the choice $F = \Omega$.
	\end{remark}
	\Cref{lem:Lipschitz_symmetric_difference} implies the 
	following lemma, which will be useful in the remainder.
	\begin{lemma}\label{lem:l1_diff_linearly_bounded_in_t_near_zero}
		Let $\{E_1,\ldots,E_M\}$ be a Caccioppoli partition of $\Omega$ and $v = \sum_{i=1}^M \nu_i \chi_{E_i}$.
		Let $(f_t)_{t\in(-\varepsilon,\varepsilon)}$ be the local variation
		defined by $f_t \coloneqq I +t \phi$ with $\phi \in C_c^\infty(\Omega, \R^d)$.
		Then there exist $\varepsilon_0 > 0$ and $C > 0$ such that 
		for all $t \in (-\varepsilon_0,\varepsilon_0)$ it holds that
		\[ \|f_t{^\#}v - v\|_{L^1} \le C|t|. \]
	\end{lemma}
	\begin{proof}
		\Cref{lem:Lipschitz_symmetric_difference} gives
		$L > 0$ and $\varepsilon_0 > 0$
		such that
		\[ \lambda(f_t(E_i) \cap E_j) = \lambda(f_t(E_i) \cap E_j) - \lambda(f_0(E_i) \cap E_j)
		\le L|t|P(E_i,\Omega) \]
		holds for all $t \in (-\varepsilon_0,\varepsilon_0)$ and
		all $i$, $j \in \{1,\ldots,M\}$.
		Because $\{E_1,\ldots,E_M\}$ and $\{f_{t}(E_1),\ldots,f_t(E_M)\}$ are 
		Caccioppoli partitions of $\Omega$, we insert the estimate above and obtain
		\begin{gather*}
			\|f_t^{\#}v - v\|_{L^1}
			= \sum_{i=1}^M\sum_{j=1}^M|\nu_i - \nu_j| \lambda(f_t(E_i) \cap E_j) 
			\le L|t| \sum_{i=1}^M P(E_i,\Omega)\sum_{j=1}^M|\nu_i - \nu_j|
			\end{gather*}
		for all $t \in (-\varepsilon_0,\varepsilon_0)$.
		We note that
		$P(E_i,\Omega) = \TV(\chi_{E_i}) = \Ha^{d-1}(\partial^* E_i \cap \Omega)
		= \sum_{j=1, j \neq i}^M \Ha^{d-1}(\partial^* E_i \cap \partial^* E_j)
		\le \TV(v) < \infty$
		because of $|\nu_i - \nu_j| \ge 1$ for $i \neq j$ and
		\eqref{eq:tv_identity}, which yields
		\[ \|f_t^{\#}v - v\|_{L^1}
		\le |t| L \TV(v) \sum_{i = 1}^M \sum_{j = 1}^M |\nu_i - \nu_j|. \]
	\end{proof}
	
	\section{Locally Optimal Solutions and First-order Optimality Conditions}\label{sec:sol}
	First we define $r$-optimality of feasible solutions for \eqref{eq:p}.
	\begin{definition}\label{dfn:roptimal}
		Let $v$ be feasible for \eqref{eq:p}.
		Then \emph{$v$ is $r$-optimal} for some $r > 0$ if
		\[ F(v) + \alpha \TV(v) \le F(\tilde{v}) + \alpha \TV(\tilde{v}) \]
		holds for all $\tilde{v}$ that are feasible for \eqref{eq:p}
		and satisfy $\|v - \tilde{v}\|_{L^1} \le r$.
	\end{definition}
	Clearly, $r$-optimality is a necessary condition
	for (global) optimality. In finite dimension, this 
	corresponds
	to so-called \emph{local minimizers} for mixed-integer
	optimization problems \cite{newby2015trust},
	where the optimal integer solution in a neighborhood is
	called a \emph{local minimizer}.
	In $\R^n$, $n \in \N$, there is always a small enough
	neighborhood around a feasible integer point
	that does not contain any further 
	feasible (integer) points, however. This is not true in our infinite-dimensional
	setting. On the contrary, the following example is
	generic.
	
	\begin{example}
		Let $M \ge 2$, 
		and $v \in \BVV(\Omega)$ with $V = [a,b] \cap \Z$, $a \leq b-1$. We may consider a ball
		$B \subset \Omega$ and
		construct $\hat{v} \in \BVV(\Omega)$ by setting
		\[ \hat{v}(x) \coloneqq \left\{
		\begin{aligned}
		\nu_{i + 1} & \text{ if } x \in B \text{ and } v(x) = \nu_i
		\text{ for } i < M, \\
		\nu_{M - 1} & \text{ if } x \in B \text{ and } v(x) = \nu_M, \\
		v(x) & \text{ if } x \in \Omega\setminus B.
		\end{aligned}
		\right.
		\]
		Then $\|v - \hat{v}\|_{L^1} = \lambda(B)$, which tends to zero
		when driving the radius of $B$ to zero.
	\end{example}
	
	In the one-dimensional case $d = 1$ a first-order optimality
	condition for $r$-optimal points follows from the variational
	argument that \emph{at a (local) minimizer $v^*$
		a small perturbation of any of the (finitely many)
		switching points of $v^*$ yields an increase of the objective}.
	Specifically, the derivative of the first term of the objective 
	with respect to the perturbation needs to be zero because
	the term $\alpha \TV$ is unaffected by small 
	perturbations of the switching locations.
	
	While there are no \emph{switching locations} for our
	case $d \ge 2$, we may consider boundaries between the 
	different level sets of
	a (local) minimizer $v^*$ instead. Thus the idea can be
	translated to $d \ge 2$ by perturbing the level sets by
	means of local variations. 
	
	We prove a first-order optimality condition for \eqref{eq:p}
	under the following assumption.
	\begin{assumption}\label{ass:hessian_regularity}
		Let
		$F : L^2(\Omega) \to \R$
		be twice continuously Fr\'{e}chet differentiable. For
		some $C > 0$ and all $\xi \in L^2(\Omega)$, let the bilinear form
		induced by the Hessian $\nabla^2 F(\xi) : L^2(\Omega)\times L^2(\Omega) \to \R$
		satisfy
		$|\nabla^2 F(\xi)(u,w)| \le C\|u\|_{L^1}\|w\|_{L^1}$
		for all $u$, $w \in L^2(\Omega)$.
	\end{assumption}
	As discussed in \cite{leyffer2021sequential},
	\cref{ass:hessian_regularity} means that $F$ improves 
	the 
	regularity of its input. Similar assumptions are present in 
	other works on discrete-valued control functions, see, e.g.,
	(5) in Lemma 3 and (10) in Theorem 2 in \cite{hahn2022binary} and 
	Assumption 3.1 3 in \cite{manns2022on}. The reason 
	is that for discrete-valued control functions $v$, $w$,
	reductions in the linear part of Taylor's expansion of $F$ 
	are only proportional to $\|w - v\|_{L^1}$. Moreover, $|V| < \infty$ implies $\|w - v\|_{L^1} = \Theta(\|w - v\|_{L^2}^2)$, 
	which gives that reductions
	in the linear part, proportional to $\|w - v\|_{L^1}$, 
	do not necessarily dominate the remainder term of the Taylor expansion.
	
	We define our concept of stationarity in \cref{dfn:l_stationarity} below.
	\begin{definition}\label{dfn:l_stationarity}
		Let $F$ satisfy \cref{ass:hessian_regularity}. Let $\{E_1,\ldots,E_M\}$
		be a Caccioppoli partition of $\Omega$,
		$v = \sum_{i=1}^M \nu_i \chi_{E_i}$, and $\nabla F(v) \in C(\bar{\Omega})$. 
		Then we say that $v$ is \emph{L-stationary}
		if the identity
		\begin{gather}\label{eq:p_variational_principle}
		\begin{multlined}
		\sum_{i=1}^M \nu_i\int_{\partial^*E_i\cap \Omega}
		(-\nabla F(v))(x)(\phi(x)\cdot n_{{E}_i}(x))\dd \Ha^{d-1}(x) = \\
		\alpha \sum_{i=1}^M \sum_{j=i + 1}^M|\nu_i - \nu_j|
		\int_{\partial^*E_i \cap \partial^* E_j}\bdvg{E_i} \phi(x)\dd \Ha^{d-1}(x).
		\end{multlined}
		\end{gather}
		holds for all $\phi \in C_c^\infty(\Omega; \R^d)$.
	\end{definition}
	\begin{remark}
		\Cref{dfn:l_stationarity}  means that on the intersection of the (essential)
		boundaries of $E_i$ and ${E}_j$,
		the set ${E}_i$ has distributional mean curvature of
		$\frac{\nu_i - \nu_j}{|\nu_i - \nu_j|}(-\nabla F)$.
	\end{remark}
	We prove our first-order optimality condition that\ $r$-optimal points are L-stationary below.
	\begin{theorem}\label{thm:P_stationarity}
		Let $F$ satisfy \cref{ass:hessian_regularity}.
		Let $\{E_1,\ldots,E_M\}$ be a Caccioppoli partition of $\Omega$,
		and $v = \sum_{i=1}^M \nu_i \chi_{E_i}$. 
		Let $\nabla F(v) \in C(\bar{\Omega})$.
		If $v$ is $r$-optimal for \eqref{eq:p} for some $r > 0$,
		then it is $L$-stationary. Morever, for all $i$, $j \in \{1,\ldots,M\}$ with $i \neq j$ we have
		\begin{gather}\label{eq:claimed_identity}
		\begin{multlined}
		(\nu_i - \nu_j)
		\int_{\partial^*E_i \cap \partial^* E_j}
		(-\nabla F(v))(x)(\phi(x) \cdot n_{E_i}(x))\dd \Ha^{d-1}(x)\\
		= \alpha |\nu_i - \nu_j| \int_{\partial^* E_i \cap \partial^* E_j} \bdvg{E_i}\phi(x)\dd\Ha^{d-1}(x).
		\end{multlined}
		\end{gather}
		for all $\phi\in C_c^\infty(E_i \cup E_j, \R^d)$.
	\end{theorem}
	\begin{proof}
		Let $(f_t)_{t\in(-\varepsilon,\varepsilon)}$ be the local variation
		defined by $f_t \coloneqq I + t \phi$ for $\phi \in C_c^\infty(\Omega; \R^d)$.
		We first prove that the function $t \mapsto F(f_t^{\#}v) + \alpha \TV(f_t^{\#}v)$ is differentiable at $t = 0$.
		We recall from \cref{lem:l1_diff_linearly_bounded_in_t_near_zero} that there exist $\varepsilon_0 > 0$ and $\bar{C} > 0$ such that
		\[ \|f_t^{\#}v - v\|_{L^1} \le
		\bar{C} |t| \TV(v)
		\]
		holds for all $t \in (-\varepsilon_0,\varepsilon_0)$.
		Let $\xi^t \in L^2(\Omega)$ be in the line segment between
		$v$ and $f^{\#}_tv$. Then \cref{ass:hessian_regularity} gives
		\[ \left|\nabla^2 F(\xi^t)(f_t^{\#}v - v, f_t^{\#}v - v)\right|
		\le C \bar{C}^2 \TV(v)^2 t^2 \]
		with the constant $C > 0$ from \cref{ass:hessian_regularity}.
		Thus the left hand side of the above estimate is differentiable
		at $t = 0$ with derivative zero. Consequently, $t \mapsto F(f_t^{\#}v)$
		is differentiable at $t = 0$ with derivative
		\[ \frac{\dd}{\dd t} F(f_t^{\#}v)\Big|_{t = 0}
		= \frac{\dd}{\dd t}
		(\nabla F(v), f_t^{\#}v)_{L^2}\Big|_{t = 0} \]
		by virtue of Taylor's theorem applied to $F$,
		where the differentiability of the right hand side
			is due to \cref{lem:Taylor_lin_mg2}.
		\Cref{lem:Taylor_TV_mg2} implies that $t \mapsto (\nabla F(v), f_t^{\#}v)_{L^2}$ is differentiable at $t = 0$ as well.
		According to \Cref{lem:Taylor_TV_mg2,lem:Taylor_lin_mg2}, the derivatives are
		\begin{align*}
		\frac{\dd}{\dd t} (\nabla F(v), f_t^{\#}v)_{L^2}\Big|_{t = 0}
		&= \sum_{i=1}^M\nu_i
		\int_{\partial^*E_i  \cap \Omega} (\phi(x) \cdot n_{E_i}(x)) \nabla F(v)(x)
		\dd \Ha^{d-1}(x) \text{ and}\\
		\frac{\dd}{\dd t} \TV(f_t^{\#}v)\Big|_{t = 0}
		&= \sum_{i=1}^M \sum_{j=i + 1}^M|\nu_i - \nu_j|
		\int_{\partial^*E_i \cap \partial^* E_j}\bdvg{E_i} \phi(x)\dd \Ha^{d-1}(x).
		\end{align*}
		
		Consequently, the function $t \mapsto F(f_t^{\#}v) + \alpha \TV(f_t^{\#}v)$
		is differentiable at $t = 0$ and we obtain the first-order optimality condition
		\[ \frac{\dd}{\dd t} F(f_t^{\#}v) + \alpha \TV(f_t^{\#}v)\Big|_{t = 0} = 0 \]
		by virtue of Fermat's theorem and the $r$-optimality of $v$.
		
		Combining these equations yields the identity
		\begin{multline*}
		\sum_{i=1}^M \nu_i
		\int_{\partial^*{E}_i \cap \Omega} (-\nabla F(v)(x))(\phi(x) \cdot n_{{E}_i}(x))
		\dd \Ha^{d-1}(x)
		= \\
		\alpha
		\sum_{i=1}^M \sum_{j=i + 1}^M|\nu_i - \nu_j|
		\int_{\partial^*E_i \cap \partial^* E_j}\bdvg{E_i} \phi(x)\dd \Ha^{d-1}(x)   
		\end{multline*}
		for all $\phi \in C^\infty_c(\Omega;\R^d)$, which is L-stationarity
		of $v$.
		Restricting to $\phi \in C_c^\infty({E}_i \cup {E}_j,\R^d)$ 
		for $i \neq j$
		and using $n_{{E}_i} = -n_{{E}_j}$ on
		$\partial^*{E}_i \cap \partial^* {E}_j$ gives the
		second claim.
	\end{proof}
	
	\section{Sequential Linear Integer Programming Algorithm}\label{sec:algorithm}
	We introduce the trust-region algorithm in $\BVV(\Omega)$ as proposed
	in \cite{leyffer2021sequential} and its trust-region subproblems.
	We recall the trust-region subproblems, analyze
	their $\Gamma$-convergence with respect to the
	linearization point, and provide optimality conditions for the
	trust-region subproblem in \cref{sec:trsub}.
	Then we introduce the algorithm in \cref{sec:slip}.
	
	\subsection{Trust-region Subproblem}\label{sec:trsub}
	We study the following trust-region subproblem, in
	which the objective of \eqref{eq:p}
	is linearized partially and the $\TV$-term is considered exactly. 
	Let $\Delta > 0$ and $\bar{v} \in \BVV(\Omega)$. We study the problem
	\begin{gather}\label{eq:tr}
	\text{{\ref{eq:tr}}}(\bar{v}, g, \Delta) \coloneqq
	\left\{
	\begin{aligned}
	\min_{v \in L^2(\Omega)}\ & (g, v - \bar{v})_{L^2} + \alpha \TV(v)-\alpha \TV(\bar{v})\\
	\text{s.t.}\quad & \|v - \bar{v}\|_{L^1} \le \Delta\text{ and }v(x) \in V \text{ for a.a.\ } x \in \Omega,
	\end{aligned}
	\right.
	\tag{TR}
	\end{gather}
	where we are interested in the case $g = \nabla F(\bar{v})$.
	The remainder of this work frequently
	uses the fact that $\text{{\ref{eq:tr}}}(\bar{v},g,\Delta)$ admits a
	minimizer, which we recap below.
	\begin{proposition}\label{prp:tr_well_defined}
		Let $\bar{v} \in \BVV(\Omega)$, $g \in L^2(\Omega)$, $\Delta \ge 0$,
		and $\TV(\bar{v}) < \infty$.
		Then $\text{\emph{\ref{eq:tr}}}(\bar{v},g,\Delta)$ admits a minimizer.
	\end{proposition}
	\begin{proof}
		A proof can be found in \cite[Prop.\,2.3]{leyffer2021sequential}.
	\end{proof}
	The analyzed algorithmic framework in function space (see \cref{sec:slip}) produces convergent subsequences of integer-valued
	control functions. In particular, we prove that the mode of convergence is not only weak-$^*$ but strict in $\BV(\Omega)$
	(see \cref{sec:trm_analysis_mg2}). 
	We can obtain $\Gamma$-convergence of the trust-region subproblems with respect to strict 
	convergence of the linearization point when the trust-region radius is kept constant.
	This will be an ingredient of the convergence proof of
	the superordinate trust-region algorithm.
	
	We believe that this result is interesting in its own right 
	because it shows a flexibility to approximate (regularize) 
	$g$ in order to improve
	the solution process of the trust-region subproblems.  Specifically, we prove the
	following $\Gamma$-convergence result.
	\begin{theorem}\label{thm:tr_gamma_convergence}
		Let $d \ge 2$. Let $v^n \to v$ strictly in $\BVV(\Omega)$.
		Let $g^n \weakto g$ in $L^2(\Omega)$.
		Let $\Delta > 0$. Then the functionals $T^n : (\BVV(\Omega), \text{\emph{weak}-$^*$}) \to \R$ defined as
		\[ 
		T^n(w) \coloneqq (g^n, w - v^n)_{L^2} + \alpha \TV(w)-\alpha \TV(v^n) +
		\delta_{[0,\Delta]}(\|w - v^n\|_{L^1})
		\]
		for $w \in \BVV(\Omega)$, $\Gamma$-converge to $T :  (\BVV(\Omega), \text{\emph{weak}-$^*$}) \to \R$, defined as
		\[
		T(w) \coloneqq (g, w - v)_{L^2} + \alpha \TV(w)-\alpha \TV(v) +
		\delta_{[0,\Delta]}(\|w - v\|_{L^1})
		\]
		for $w \in \BVV(\Omega)$, where $\delta_{[0,\Delta]}$ is the $\{0,\infty\}$-valued indicator function of $[0,\Delta]$.
	\end{theorem}
	\begin{proof}
		We start with the lower bound inequality, that is we need to show $T(w) \le \liminf_{n \to \infty} T^n(w^n)$
		for $w^n \weakstarto w$ in $\BVV(\Omega)$. Then $w^n \to w$ in $L^2(\Omega)$ and also
		$v^n \to v$ in $L^2(\Omega)$, implying $(g^n, w^n - v^n)_{L^2} \to (g, w - v)_{L^2}$.
		The strict convergence of $(v^n)_n$ and the weak-$^*$ convergence of $(w^n)_n$ imply
		\[ \alpha \TV(w) - \alpha \TV(v) \le \liminf_{n \to \infty} \alpha \TV(w^n) - \alpha \TV(v^n)  \]
		because $\TV$ is lower semi-continuous with respect to weak-$^*$ convergence in $\BV(\Omega)$.
		Thus the lower bound inequality holds true if 
		(after restricting to a subsequence) the implication
		\[ 
		\|w^n - v^n\|_{L^1} \le \Delta \text{ for all } n \in \N
		\quad\Longrightarrow\quad
		\|w - v\|_{L^1} \le \Delta
		\]
		holds. This follows from
		the triangle inequality, $w^n \to w$ in $L^1(\Omega)$, and $v^n \to v$ in $L^1(\Omega)$.
		
		Next, we need to show the upper bound inequality, that is we need to show that for each $w \in \BVV(\Omega)$
		there exists a sequence $w^n \weakstarto w$ in $\BVV(\Omega)$ such that
		$T(w) \ge \limsup_{n \to \infty} T^n(w^n)$. We distinguish
		three cases for the value of $\|w - v\|_{L^1}$.
		
		\textbf{Case $\|w - v\|_{L^1} > \Delta$:} Then $T(v) = \infty$ and $w^n \coloneqq w$ for all $n \in \N$
		give the claim.
		
		\textbf{Case $\|w - v\|_{L^1} < \Delta$:} Let $w^n \coloneqq w$ for all
		$n \in \N$. Then
		$(g^n, w^n - v^n)_{L^2} + \alpha \TV(w^n) - \alpha \TV(v^n) \to (g, w - v)_{L^2} + \alpha \TV(w) - \alpha \TV(v) + \delta_{[0,\Delta]}(\|w - v\|_{L^1})$.
		Moreover, the triangle inequality and $v^n \to v$ in $L^1(\Omega)$ imply
		$\|w - v^n\|_{L^1} < \Delta$ for all $n \in \N$ large enough and thus
		$\|w^n - v^n\|_{L^1} < \Delta$ because $w^n = w$. Combining these assertions
		we obtain $T^n(w^n) \to T(w)$.
		
		\textbf{Case $\|w - v\|_{L^1} = \Delta$:} The fact $\Delta > 0$ implies that there exists a set
			\[ D \coloneqq \{ x \in \Omega\,|\, v(x) = \nu_1 \neq \nu_2 = w(x) \}
			\]
			with $\lambda(D) > 0$, where the specific control realizations $\nu_1$ and $\nu_2$
			are without loss of generality because we may reorder the indices of the
			elements of $V$ if necessary.
			Moreover, $D$ is a set of finite perimeter because
			\begin{multline*}
			P(D,\Omega) 
			= P(v^{-1}(\{\nu_1\}) \cap w^{-1}(\{\nu_2\}),\Omega)
			\le \\ P(v^{-1}(\{\nu_1\}),\Omega) + P(w^{-1}(\{\nu_2\}),\Omega) \underset{\eqref{eq:Ek_has_finite_perimeter}}\le 2\TV(v) + 2\TV(w) < \infty,
			\end{multline*}
			where the first inequality follows from \cite[Prop.\ 3.38]{ambrosio2000functions}. Because $\lambda(D)>0$, there exists some point $\bar{x} \in D$ of density~$1$, i.e.,
			\begin{align*}
			\lim_{r \searrow 0} \frac{\lambda(D \cap B_r(\bar{x}))}{\lambda(B_r(\bar{x}))} = 1.
			\end{align*}
			This implies that there exists a monotonically decreasing sequence $(r_k)_{k \in \N}$ such that $0 < l_k \coloneqq \lambda(D \cap B_{r_k}(\bar{x}))$ defines a monotonically decreasing sequence $(l_k)_{k \in \N}$ with $l_k \searrow 0$ as $k \to \infty$. Since $v^n \to v$ in $L^1(\Omega)$, we deduce that there exist $N_k \in \N$
			such that
			\begin{gather*}
			\begin{aligned}
			\|v - v^n \|_{L^1} \leq l_k &\quad \text{for all } n \geq N_k \text{ and} \\
			N_k < N_{k+1}
			\end{aligned}    
			\end{gather*}
			hold for all $k \in \N$.
			Let $n \in \N$ with $n \geq N_1$. Then $n \in [N_k,N_{k+1})$ for some $k \in \N$ and we define
			\[ w^n(x) \coloneqq \left\{ 
			\begin{aligned}
			v(x) &\text{ if } x \in D \cap B_{r_k}(\bar{x}),\\
			w(x) &\text{ else}
			\end{aligned}
			\right.
			\]
			for all $x \in \Omega$. We deduce
			\[ \|w^n - v^n\|_{L^1}
			\le \|v - v^n\|_{L^1}
			+ \|w^n - v\|_{L^1}
			\le l_k + \Delta - |\nu_1 - \nu_2|l_k \le \Delta, \]
			where we have used that $|\nu_1 - \nu_2| \ge 1$.
			The construction gives $w^n \weakstarto w$ in 
			$\BVV(\Omega)$.
			It remains to show $\TV(w^n) \to \TV(w)$.
			
			To see this let $E_i \coloneqq w^{-1}(\{\nu_i\})$, $E_i^n \coloneqq (w^n)^{-1}(\{\nu_i\})$
			for $i \in \{1,\ldots,M\}$ and $n \in \N$, and $D^n \coloneqq D \cap B_{r_k}(\bar{x})$ for $n \in \N$
			and the corresponding $k$ that depends on $n$ as above. There hold $E_1^n = E_1 \cup D^n$, $E_2^n = E_2 \setminus D^n$, $E_i^n = E_i$ for $i \geq 3$, and therefore also
			\begin{align*}
			\partial^* E_1^n \subset\partial^e E_1 \cup \partial^e D^n \quad \text{ and }
			\quad \partial^* E_2^n \subset \partial^e E_2 \cup \partial^e D^n,
			\end{align*}
			where the inclusions follow from \Cref{lem:boundary_union}.
			This yields
			\begin{align*}
			\partial^* E_1^n \cap \partial^* E_2^n \subset (\partial^e E_1 \cup \partial^e D^n) \cap (\partial^e E_2 \cup \partial^e D^n ) = (\partial^e E_1 \cap \partial^e E_2) \cup \partial^e D^n
			\end{align*}
			and, analogously,
			\begin{align*}
			\partial^* E_1^n \cap \partial^* E_i^n
			\subset (\partial^e E_1 \cap \partial^e E_i) \cup \partial^e D^n
			\enskip\text{and}\enskip
			\partial^* E_2^n \cap \partial^* E_i^n
			\subset (\partial^e E_2 \cap \partial^e E_i) \cup \partial^e D^n
			\end{align*}
			for $i \geq 3$. We deduce
			\begin{align*}
			\Ha^{d-1}(\partial^* E_1^n \cap \partial^* E_2^n) & \leq \Ha^{d-1}((\partial^e E_1 \cap \partial^e E_2) \cup \partial^e D^n ) \\
			& = \Ha^{d-1}((\partial^*E_1 \cap \partial^* E_2 ) \cup \partial^* D^n ) 
			&& {\text{\small{\cite[Thm. 3.61]{ambrosio2000functions}}}}
			\\
			& \leq \Ha^{d-1}(\partial^*E_1 \cap \partial^* E_2 ) + \Ha^{d-1}(\partial^* D^n)
			\end{align*}
			and, analogously,
			\begin{align*}
			\Ha^{d-1}(\partial^* E_1^n \cap \partial^* E_i^n) \leq \Ha^{d-1}(\partial^*E_1 \cap \partial^* E_i ) + \Ha^{d-1}(\partial^* D^n)
			\end{align*}
			and
			\begin{align*}
			\Ha^{d-1}(\partial^* E_2^n \cap \partial^* E_i^n) \leq \Ha^{d-1}(\partial^*E_2 \cap \partial^* E_i ) + \Ha^{d-1}(\partial^* D^n).
			\end{align*}
			Then
			\begin{align*}
			\TV(w^n) &\underset{\eqref{eq:tv_identity}}= \sum_{i = 1}^M\sum_{\ell = i + 1}^M |\nu_i - \nu_\ell|\Ha^{d-1}(\partial^* E_i^n \cap \partial^* E_\ell^n)\\
			& \leq \sum_{i=1}^2 \sum_{\ell = i+1}^M | \nu_i - \nu_\ell | (\Ha^{d-1}(\partial^* E_i \cap \partial^* E_{\ell} ) + \Ha^{d-1}(\partial^* D^n)) \\
			& + \sum_{i=3}^M \sum_{\ell = i+1}^M | \nu_i - \nu_\ell | \Ha^{d-1}(\partial^* E_i \cap \partial^* E_{\ell} ) \\
			& = \TV(w) + \sum_{i=1}^2 \sum_{\ell = i+1}^M | \nu_i - \nu_\ell | \Ha^{d-1}(\partial^* D^n) \\
			& \leq \TV(w) + 2 M \nu_{\max} \, \Ha^{d-1}(\partial^*D^n) \\
			& \leq \TV(w) + 2 M \nu_{\max} \, (\Ha^{d-1}(\partial B_{r_k}(\bar{x})) + \Ha^{d-1} \mres \partial^* D ( B_{r_k}(\bar{x})))
			\end{align*}
			with $\nu_{\max} \coloneqq \max_{i,\ell \in \{1,\dots,M\}} |\nu_i - \nu_{\ell} | \geq 1$ and where the last inequality follows from (15.15) in \cite{maggi2012sets}.
			Because the measure $\Ha^{d-1} \mres \partial^* D$ is a finite Radon measure, the term
			$\Ha^{d-1} \mres \partial^* D ( B_{r_k}(\bar{x}))$ tends to zero as $n$
			and the corresponding $k$ tend to infinity. Moreover, $\Ha^{d-1}(\partial B_{r_k}(\bar{x}))$ tends to zero as well.
			Together with the lower semicontinuity of $\TV$ we obtain $\TV(w^n) \to \TV(w)$.
	\end{proof}
	\begin{corollary}
		Let $F : L^2(\Omega) \to \R$ be continuously Fr\'{e}chet differentiable. Then the claimed $\Gamma$-convergence holds
		with the choice $g^n = \nabla F(v^n)$.
	\end{corollary}
	\begin{proof}
		This follows from \cref{thm:tr_gamma_convergence} because $v^n \to v$ strictly in $\BVV(\Omega)$
		implies that $v^n \to v \in L^2(\Omega)$, which in turn implies $\nabla F(v^n) \to \nabla F(v^n)$ in $L^2(\Omega)$.
	\end{proof}
	The proof of the lower bound inequality of \cref{thm:tr_gamma_convergence} 
	applies for the case $d = 1$ as well. The proof of the upper bound inequality
	uses $\Ha^{d-1}(\partial B_{r_k}(x)) \searrow 0$ for $r_k \searrow 0$,
	which requires $d > 1$. In fact, the
	claim is not true for $d = 1$, as is demonstrated below.
	
	\begin{example}\label{ex:tr_ubie}
		Let $\Omega = (-2,2) \subset \R$. Let $V = \{0, 1\}$.
		$w = 0_{BV(\Omega)}$.
		Let $v = \chi_{[-1,1]}$. Let $\Delta = 2$. Let $v^n = \chi_{\left[-1-1/n,1+1/n\right]}$.
		Then $v^n \to v$ strictly in $\BVV(\Omega)$.
		Let $g^n \weakto g$ in $L^2(\Omega)$. Let
		$T^n$ for $n \in \N$, $T$ be defined as in \cref{thm:tr_gamma_convergence}.
		Moreover, $\|v - w\|_{L^1} = \Delta$ and
		$\|v^n - w\|_{L^1} = 2 + 2/n > \Delta$.
		We need to approximate $w$ by some $w^n$ such that
		$\|v^n - w^n\|_{L^1} \le \Delta$. Let
		$\delta^n \coloneqq 2/n$.
		
		The smallest value $\TV(w^n)$ can attain in this case
		is $\TV(w^n) = 2$---by choosing, for example, $w^n = \chi_{[-\delta^n/2,\delta^n/2]}$---because $\TV(w^n) = 0$
		if and only if $w^n$ is constant on all of $\Omega$.
		We obtain 
		$\limsup_{n\to\infty} T^n(w^n) \ge 2\alpha > 0 = \alpha \TV(w)$ for any sequence $w^n \weakstarto w$ in $\BV(\Omega)$
		such that $\|v^n - w^n\|_{L^1} \le \Delta$.
		Thus, the upper bound inequality is violated.
	\end{example}
	
	Next, we apply \cref{thm:P_stationarity} to $\text{\ref{eq:tr}}(v,g,\Delta)$, which also asserts
	that $v$ is L-stationary as well if it solves
	$\text{{\ref{eq:tr}}}(v,\nabla F(v),\Delta)$.
	\begin{proposition}\label{prp:TRset_stationarity}
		Let $g \in C(\bar{\Omega})$. Let $\Delta > 0$.
		Let $\{E_1,\ldots,E_M\}$ be a Caccioppoli partition of
		$\Omega$, and $v = \sum_{i=1}^M \nu_i \chi_{{E}_i}$.
		If $v$ is $r$-optimal for $\text{\emph{\ref{eq:tr}}}(v,g,\Delta)$ for some
		$r > 0$, then for all $\phi \in C_c^\infty(\Omega;\R^d)$ it holds that
		\begin{gather*}
		\begin{multlined}
		\sum_{i=1}^M \nu_i\int_{\partial^*{E}_i \cap \Omega} (-g)\phi\cdot n_{{E}_i}\dd \Ha^{d-1}
		= \alpha \sum_{i=1}^M \sum_{j=i + 1}^M|\nu_i - \nu_j|
		\int_{\partial^*E_i \cap \partial^* E_j}\bdvg{E_i} \phi\dd \Ha^{d-1}
		\end{multlined}
		\end{gather*}
		In particular, it holds for all $i$, $j \in \{1,\ldots,M\}$ with
		$i \neq j$ that
		\begin{gather*}
		\begin{multlined}
		(\nu_i - \nu_j) \int_{\partial^*{E}_i \cap \partial^* E_j} (-g)\phi\cdot n_{{E}_i}\dd \Ha^{d-1}
		= \alpha |\nu_i - \nu_j|
		\int_{\partial^*E_i \cap \partial^* E_j}\bdvg{E_i} \phi\dd \Ha^{d-1}
		\end{multlined}
		\end{gather*}
		for all $\phi\in C_c^\infty(E_i \cup E_j, \R^d)$.
	\end{proposition}
	\begin{proof}
		If $v$ is $r$-optimal for $\text{{\ref{eq:tr}}}(v,g,\Delta)$ for some
		$r > 0$, then it is $r_0$-optimal for some $r_0 \le \Delta$,
		implying that only feasible points for $\text{{\ref{eq:tr}}}(v,g,\Delta)$
		are considered.
		We choose $F(u) \coloneqq (g,u)_{L^2}$, $u \in L^2(\Omega)$, which gives
		$\nabla F(u) = g$
		and $\nabla^2 F(u) = 0$ and in turn that 
		\cref{ass:hessian_regularity} is satisfied.
		Then the claim follows from \cref{thm:P_stationarity}
		with $r_0$ for $r$.
	\end{proof}
	
	\subsection{Algorithm Statement}\label{sec:slip}
	We propose to solve \eqref{eq:p} for $r$-optimal points
	or stationary points with \cref{alg:trm} \cite{leyffer2021sequential}
	below. The algorithm iterates over two nested loops.
	An outer iteration completes (the inner loop terminates) when a 
	new iterate has been
	computed successfully or the optimal objective value
	of the trust-region subproblem \eqref{eq:tr} is zero.
	In the latter case the algorithm terminates, else the trust-region
	radius is reset and the inner loop is triggered again.
	The inner loop solves \eqref{eq:tr}
	for shrinking trust-region radii
	until the predicted reduction,
	the negative objective of \eqref{eq:tr}, is less or equal than zero (necessary condition for $r$-optimality)
	or a sufficient decrease condition is met.
	The latter means that the solution of \eqref{eq:tr} is accepted as
	the next iterate. To this end, we
	use the acceptance criterion (sufficient decrease condition) \cite{manns2022on}
	\begin{gather}\label{eq:accept}
	\ared(v^{n-1},\tilde{v}^{n,k}) \ge \sigma \pred(v^{n-1},\Delta^{n,k})
	\end{gather}
	for some $\sigma \in (0,1)$, where
	\[ \ared(v^{n-1},\tilde{v}^{n,k})
	\coloneqq F(v^{n-1}) + \alpha \TV(v^{n-1})
	- F(\tilde{v}^{n,k}) - \alpha \TV(\tilde{v}^{n,k})
	\]
	is the reduction achieved by the solution $\tilde{v}^{n,k}$
	of the trust-region subproblem, and 
	\[ \pred(v^{n-1},\Delta^{n,k})
	\coloneqq (\nabla F(v^{n-1}),v^{n - 1} - \tilde{v}^{n,k})
	+ \alpha \TV(v^{n-1}) - \alpha \TV(\tilde{v}^{n,k})
	\]
	is the predicted reduction by the (negative objective of the) trust-region 
	subproblem for the current trust-region radius and thus its solution 
	$\tilde{v}^{n,k}$.
	
	\begin{algorithm}[t]
		\caption{Sequential linear integer programming method (SLIP)}\label{alg:trm}
		\textbf{Input:} $F$ sufficiently regular, $\Delta^0 > 0$, $v^0 \in \BVV(\Omega)$, $\sigma \in (0,1)$.
		
		\begin{algorithmic}[1]
			\For{$n = 0,\ldots$}
			\State $k \gets 0$
			\State $\Delta^{n,0} \gets \Delta^0$
			\While{not sufficient decrease according to \eqref{eq:accept}}\label{ln:suffdec}
			\State\label{ln:trstep} $\tilde{v}^{n,k} \gets$ minimizer of \eqref{eq:tr}
			with $\Delta = \Delta^{n,k}$, $v = v^{n-1}$, and $g = \nabla F(v^{n-1})$.
			\State\label{ln:pred} $\pred(v^{n-1},\Delta^{n,k}) \gets (\nabla F(v^{n-1}), v^{n-1} - \tilde{v}^{n,k})_{L^2}
			+ \alpha \TV(v^{n-1}) - \alpha \TV(\tilde{v}^{n,k})$
			\State $\ared(v^{n-1},\tilde{v}^{n,k}) \gets F(v^{n-1}) + \alpha \TV(v^{n-1})
			- F(\tilde{v}^{n,k}) - \alpha\TV(\tilde{v}^{n,k})$	
			\If{$\pred(v^{n-1},\Delta^{n,k}) \le 0$}
			\State Terminate. The predicted reduction for $v^{n-1}$
			is zero.
			\ElsIf{not sufficient decrease according to \eqref{eq:accept}}
			\State $k \gets k + 1$		
			\State $\Delta^{n,k} \gets \Delta^{n,k-1} / 2$.
			\Else
			\State $v^n \gets \tilde{v}^{n,k}$
			\EndIf
			\EndWhile	
			\EndFor
		\end{algorithmic}
	\end{algorithm}
	
	\section{Asymptotics of \cref{alg:trm}}\label{sec:trm_analysis_mg2}
	
	We analyze the asymptotics of the function space algorithm
	\cref{alg:trm} under \cref{ass:hessian_regularity},
		see also \cite[Ass.\ 4.1]{leyffer2021sequential}.
	The Hessian regularity assumed therein is already
	required for \cref{thm:P_stationarity}.
	We analyze the inner loop
	in \cref{sec:inner} and the outer loop in \cref{sec:outer},
	which gives that all cluster points produced by
	\cref{alg:trm} are L-stationary.
	
	\subsection{Analysis of the Inner Loop of \cref{alg:trm}}\label{sec:inner}
	We use local variations to obtain
	a sufficient decrease in the trust-region subproblem that eventually
	implies acceptance of a step in case of violation of L-stationarity.
	In light of trust-region methods this can be interpreted
	as local variations providing Cauchy points for the $\phi$ that violate L-stationarity.
	We start with a preparatory lemma about feasibility.
	
	\begin{lemma}\label{lem:tr_feasibility}
		Let $g \in C(\bar{\Omega})$, and $\Delta > 0$.
		Let $\{E_1,\ldots,E_M\}$ be a Caccioppoli partition of $\Omega$,
		and $v = \sum_{i=1}^M \nu_i \chi_{E_i}$ with $\TV(v) < \infty$.
		Let $(f_t)_{t\in(-\varepsilon,\varepsilon)}$ be the
		local variation defined by $f_t \coloneqq I+t\phi$ for $\phi \in C_c^\infty(\Omega; \R^d)$.
		Then there exist $\varepsilon_1 > 0$ such that $f_t^{\#}v$
		is feasible for $\text{\emph{\ref{eq:tr}}}(v,g,\Delta)$ for all $t \in (-\varepsilon_1,\varepsilon_1)$.
	\end{lemma}
	\begin{proof}
		We consider $\varepsilon_0 > 0$ and $C > 0$
		that are asserted in 
		\cref{lem:l1_diff_linearly_bounded_in_t_near_zero}. We
		choose $\varepsilon_1 \coloneqq \min\{{\varepsilon}_0, \Delta/{C}\}$.
		Then $f_t^{\#}v$ is feasible for $\text{{\ref{eq:tr}}}(v,g,\Delta)$
		for all $t\in(-\varepsilon_1,\varepsilon_1)$.
	\end{proof}
	
	In the following lemma, we leave out the outer iteration index $n$ for better clarity, that is we abbreviate $v \coloneqq v^{n-1}$ and $\tilde{v}^k \coloneqq \tilde{v}^{n,k}$.
	\begin{lemma}\label{lem:step_accept}
		Let \cref{ass:hessian_regularity} hold.
		Let $\sigma \in (0,1)$, $v = \sum_{i=1}^M \nu_i \chi_{E_i}$
		for a Caccioppoli partition $\{E_1,\ldots,E_M\}$ of $\Omega$,
		$\nabla F(v) \in C(\bar{\Omega})$, $\Delta^k \searrow 0$,
		and $\tilde{v}^k$ minimize $\text{\emph{\ref{eq:tr}}}(v,\nabla F(v),\Delta^{k})$ for $k \in \N$.
		Then at least one of the following statements holds true.
		\begin{enumerate}
			\item\label{itm:step_accept_v_Lstationary} The function $v$ is L-stationary.
			\item\label{itm:step_accept_v_optimal} There exists $k_0 \in \N$ such that
			the objective of $\text{\emph{\ref{eq:tr}}}(v,\nabla F(v), \Delta^{k_0})$
			at $v^{k_0}$ is zero.
			\item\label{itm:step_accept_v_sufficient_decrease} There exists $k_0 \in \N$
			such that \eqref{eq:accept} holds, that is
			$\ared(v, v^{k_0})
			\ge \sigma \pred(v,\Delta^{k_0})$.
		\end{enumerate}
	\end{lemma}
	\begin{proof}
		Let $k \in \N$. Then the objective of $\text{{\ref{eq:tr}}}(v,\nabla F(v), \Delta^{k})$
		evaluated at $v$ is zero. Consequently, if the objective value of $\tilde{v}^k$,
		which is optimal, is zero (Outcome \ref{itm:step_accept_v_optimal})
		it follows that $v$ is a minimizer of
		$\text{{\ref{eq:tr}}}(v,\nabla F(v), \Delta^{k})$
		and \cref{prp:TRset_stationarity} implies
		that $v$ is L-stationary (Outcome \ref{itm:step_accept_v_Lstationary}).
		Therefore, in order to prove the claim, it remains to show that Outcome 
		\ref{itm:step_accept_v_sufficient_decrease} holds under the assumption
		that Outcome \ref{itm:step_accept_v_Lstationary} does not.
		We abbreviate
		\[ \delta^k \coloneqq \ared(v,\tilde{v}^k)
		\quad\text{and}\quad
		p^k\coloneqq \pred(v,\Delta^k).
		\]
		It remains to prove that there exists some $k_0 \in \N$ such that
		$\delta^{k_0} \ge \sigma p^{k_0}$.
		
		Because $v$ is not L-stationary, there exist
		$\phi \in C_c^\infty(\Omega;\R^d)$ with
		$\supp \phi \neq \emptyset$ and $\eta > 0$ such that
		\[
		\begin{multlined}
		\sum_{i=1}^M \nu_i\int_{\partial^*{E}_i \cap \Omega} (-\nabla F(v))(x)(\phi(x)\cdot n_{{E}_i}(x))\dd \Ha^{d-1}(x)\\
		- \alpha \sum_{i=1}^M \sum_{j=i + 1}^M|\nu_i - \nu_j|
		\int_{\partial^*E_i \cap \partial^* E_j}\bdvg{E_i} \phi(x)\dd \Ha^{d-1}(x)
		> \eta. 
		\end{multlined}
		\]
		Let $(f_t)_{t\in(-\varepsilon,\varepsilon)}$ be the local variation defined by $f_t \coloneqq I + t \phi$. \Cref{lem:tr_feasibility} implies that there exists
		$\varepsilon^k \in (0,\varepsilon)$ such that $f_t^{\#} v$ is feasible for
		$\text{{\ref{eq:tr}}}(v,\nabla F(v),\Delta^k)$ for all $t \in [0,\varepsilon^k]$. Moreover,
		the construction of $\varepsilon_1$ in the proof of \cref{lem:tr_feasibility}
		gives that we may always choose $\varepsilon^k > 0$ so that
		$\|f^{\#}_{\varepsilon^k}v - v\|_{L^1} = \Delta^k$ holds for all 
		$k$ sufficiently large.
		
		\Cref{ass:hessian_regularity} allows us to apply Taylor's theorem to $F(\tilde{v}^k)$. We deduce
		\[ \delta^k
		= p^k + \frac{1}{2}\nabla^2 F(\xi^k)(\tilde{v}^k - v, \tilde{v}^k - v)
		\ge p^k  - \frac{C}{2}\| \tilde{v}^k - v\|_{L^1}^2
		\]
		for some $\xi^k \in L^2(\Omega)$ in the line segment between $\tilde{v}^k$ and $v$
		and the $C > 0$ from \cref{ass:hessian_regularity}.
		We deduce for $t \in (0,\varepsilon^k]$
		\begin{align*}
		\delta^k &\ge \sigma p^k
		-(1- \sigma)\left((\nabla F(v), \tilde{v}^k - v)_{L^2} + \alpha \TV(\tilde{v}^k) - \alpha \TV(v)\right)
		- \frac{C}{2}\| \tilde{v}^k - v\|_{L^1}^2 \\
		&\ge \sigma p^k -(1 - \sigma)\left(
		(\nabla F(v), f_{\varepsilon^k}^{\#}v - v)_{L^2} + \alpha \TV(f_{\varepsilon^k}^{\#}v) - \alpha \TV(v)\right)
		- \frac{C}{2}\| \tilde{v}^k - v\|_{L^1}^2 \\
		&\ge \sigma p^k + (1-\sigma)\left(\varepsilon^k \eta + o(\varepsilon^k)\right) - \frac{C}{2}\|\tilde{v}^k - v\|_{L^1}^2 \\
		&\ge \sigma p^k + (1-\sigma)\left(\varepsilon^k \eta + o(\varepsilon^k)\right) - \frac{C}{2}\kappa^2 (\varepsilon^k)^2
		\end{align*}
		for some $\kappa > 0$. In particular, the second inequality follows from the optimality of
		$\tilde{v}^k$ for $\text{{\ref{eq:tr}}}(v,\nabla F(v), \Delta^{k})$,
		the third inequality follows from 
		\cref{lem:Taylor_TV_mg2,lem:Taylor_lin_mg2}, and the fourth from
			\[  -\|\tilde{v}^k - v\|_{L^1}^2
			\ge -(\Delta^k)^2 = -\|f_{\varepsilon_k}^{\#}v - v\|_{L^1}^2 \ge - \kappa^2 \varepsilon_k^2,
			\]
			where the existence of $\kappa$ is asserted by
			\cref{lem:l1_diff_linearly_bounded_in_t_near_zero}. Because $(1-\sigma)\varepsilon^k \eta > 0$ and $(1-\sigma) \varepsilon^k \eta$ eventually dominates
			the terms $(1-\sigma)o(\varepsilon^k)$ and $-\tfrac{C}{2}\kappa^2 (\varepsilon^k)^2$, there exists
			$k_0 \in \N$ such that $\delta^{k_0} \ge \sigma p^{k_0}$.
	\end{proof}
	
	\begin{corollary}\label{cor:step_accept}
		Let \cref{ass:hessian_regularity} hold. Let $v^{n-1}$ produced by \cref{alg:trm} 
		satisfy $\nabla F(v^{n-1}) \in C(\bar{\Omega})$. Then iteration
		$n$ satisfies one of the following outcomes.
		\begin{enumerate}
			\item The inner loop terminates after finitely many
			iterations and 
			\begin{enumerate}
				\item the sufficient decrease condition \eqref{eq:accept} is satisfied or
				\item the predicted reduction
				is zero (and the iterate $v^{n-1}$ is L-stationary).
			\end{enumerate}
			\item The inner loop does not terminate, and the iterate $v^{n-1}$ 
			is L-stationary.
		\end{enumerate}
	\end{corollary}
	\begin{proof}
		We apply \cref{lem:step_accept} with the choices $\Delta^k = \Delta^{n,k}$ and $v = v^{n-1}$.
	\end{proof}
	
	\subsection{Analysis of the Outer Loop}\label{sec:outer}
	With these preparations we are able to prove that the limits of the sequence
	of iterates are L-stationary under \cref{ass:hessian_regularity}.
	\begin{theorem}\label{thm:pure_tr_asymptotics}
		Let $F$ be bounded below.
		Let \cref{ass:hessian_regularity} hold. Let the iterates $(v^n)_n$ be produced by \cref{alg:trm}.
		Let $\nabla F(v^n) \in C(\bar{\Omega})$ for all $n \in \N$.
		Then all iterates are feasible for \eqref{eq:p} and 
		the sequence of objective values $(J(v^n))_n$
		is monotonically decreasing.
		Moreover, one of the following mutually exclusive outcomes holds:
		\begin{enumerate}
			\item\label{itm:finite_seq_tr} The sequence $(v^n)_n$ is finite.
			The final element $v^N$ of $(v^n)_n$ solves the
			trust-region subproblem $\text{\emph{\ref{eq:tr}}}(v^N,\nabla F(v^N),\Delta)$
			for some $\Delta > 0$ and is L-stationary.
			\item\label{itm:finite_seq_tr_contract} The sequence 
			$(v^n)_n$ is finite and the inner loop does not terminate
			for the final element $v^N$, which is L-stationary.
			\item\label{itm:infinite_seq_sl} The sequence $(v^n)_n$ has a weak-$^*$ accumulation
			point in $\BV(\Omega)$. Every weak-$^*$ accumulation point of $(v^n)_n$
			is feasible, and strict. If $v$ is a weak-$^*$ accumulation point of
			$(v^n)_n$ that satisfies $\nabla F(v) \in C(\bar{\Omega})$, then it is L-stationary.
			
			If the trust-region radii are bounded away from zero for a subsequence $(v^{n_\ell})_\ell$, that is,
			if $0 < \underline{\Delta} \coloneqq \liminf_{n_\ell\to\infty} \min_{k}\Delta^{n_{\ell}+1,k}$
			and $\bar{v}$ is a weak-$^*$ accumulation point of $(v^{n_\ell})_\ell$ with $\nabla F(\bar{v}) \in C(\bar{\Omega})$, then
			$\bar{v}$ solves $\text{\emph{\ref{eq:tr}}}(\bar{v},\nabla F(\bar{v}),\underline{\Delta} / 2)$.
		\end{enumerate}
	\end{theorem}
	\begin{proof}
		Our preparations allows us to reuse parts of the proof
		of the case $d = 1$ for the case
		$d \ge 2$ without any change other than the definition of L-stationarity. We summarize these
		parts of the proof briefly and elaborate on the arguments that differ in this work.
		
		The facts that \cref{alg:trm} produces a sequence of feasible iterates $(v^n)_n$ with corresponding
		montonotically decreasing sequence of objective function values $(J(v^n))_n$ follow exactly as
		for the case $d = 1$, see \cite[Proof of Theorem 4.23]{leyffer2021sequential}. As in
		\cite[Proof of Theorem 4.23]{leyffer2021sequential} we may restrict to the case that
		Outcomes \ref{itm:finite_seq_tr} and \ref{itm:finite_seq_tr_contract}
		do not hold true and prove Outcome \ref{itm:infinite_seq_sl} in this case
		(substitute \cite[Lemma 4.19]{leyffer2021sequential} by \cref{lem:step_accept} in the respective
		argument). We split the proof that Outcome \ref{itm:infinite_seq_sl} holds into four parts.
		
		\textbf{Outcome \ref{itm:infinite_seq_sl} (1) existence and feasibility of weak-$^*$ accumulation points:}
		This follows exactly as in \cite[Proof of Theorem 4.23]{leyffer2021sequential} using that $\BVV(\Omega)$ is
		sequentially closed in the weak-$^*$ topology of $\BV(\Omega)$.
		
		\textbf{Outcome \ref{itm:infinite_seq_sl} (2) weak-$^*$ accumulation points are strict:}
		We follow the idea of a contradictory argument from
		\cite[Proof of Theorem 4.23]{leyffer2021sequential}
		and assume that there exists a weak-$^*$ accumulation point $v$ of $(v^n)_n$ with $v^{n_\ell} \weakstarto v$
		in $\BVV(\Omega)$ such that $\TV(v) < \liminf_{\ell \to \infty} \TV(v^{n_\ell})$, that is the
		convergence is not strict. We define $\delta \coloneqq \tfrac{1}{2} (\liminf_{\ell \to \infty}\TV(v^{n_\ell}) - \TV(v))$
		but cannot assume the inequality $\delta \ge \tfrac{1}{2}$ for our case $d \ge 2$.
		By virtue of continuity of $F$ and $\nabla F$ (both following from \cref{ass:hessian_regularity})
		with respect to convergence in $L^2(\Omega)$ (thus also to weak-$^*$ convergence in $\BV(\Omega)$ in the closed subset $\BVV(\Omega)$)
		and the fact that $\Delta^{n_\ell+1,k} \to 0$
		for $k \to \infty$ (independently of $\ell$) we obtain that there exists $\ell_0 \in \N$
		and $k_0 \in \N$ such that for all $\ell \ge \ell_0$ and $k \ge k_0$
		\begin{align}
		|(\nabla F(v^{n_\ell}), v^{n_\ell} - \tilde{v})_{L^2}| &\le \frac{1-\sigma}{3-\sigma}\alpha\delta\quad\text{and}\quad
		|F(v^{n_\ell}) - F(\tilde{v})| \le \frac{1-\sigma}{3-\sigma}\alpha\delta
		\label{eq:pred_estimate_wrt_sigma_and_delta}
		\end{align}
		hold for all $\tilde{v}$ that are feasible for $\text{{\ref{eq:tr}}}(v^{n_\ell},\nabla F(v^{n_\ell}),\Delta^{n_\ell + 1,k})$.
		
		Moreover, the sufficient decrease condition \eqref{eq:accept} and the optimality of $\tilde{v}^{n_\ell + 1,k}$
		for $\text{{\ref{eq:tr}}}(v^{n_\ell},\nabla F(v^{n_\ell}),\Delta^{n_\ell + 1,k})$,
		see \cref{alg:trm} ln.\ \ref{ln:trstep}, implies that there exists $\ell_1 \ge \ell_0$ such that for all $\ell \ge \ell_1$ the estimate
		\begin{align*}
		\pred(v^{n_\ell},\Delta^{n_\ell + 1,k})\,
		&= (\nabla F(v^{n_\ell}), v^{n_\ell} - \tilde{v}^{n_\ell + 1,k})_{L^2}
		+\alpha \TV( v^{n_\ell}) -  \alpha \TV(\tilde{v}^{n_\ell + 1,k})\\
		&\ge (\nabla F(v^{n_\ell}), v^{n_\ell} - v)_{L^2}
			+\alpha \TV( v^{n_\ell}) -  \alpha \TV(v)\\
		&\ge (\nabla F(v^{n_\ell}), v^{n_\ell} - v)_{L^2}
		+\alpha \delta \\
		&\ge - \frac{1-\sigma}{3-\sigma}\alpha\delta + \alpha\delta
		\end{align*}
		holds if the iterate is accepted in inner iteraton $k$ of outer iteration $\ell$ and $v$ is feasible 
		for $\text{{\ref{eq:tr}}}(v^{n_\ell},\nabla F(v^{n_\ell}),\Delta^{n_\ell + 1,k})$.
		Note that the optimality of $\tilde{v}^{n_\ell + 1,k}$ gives the
		first inequality and \eqref{eq:pred_estimate_wrt_sigma_and_delta}
		gives the third.
		
		Because $v^{n_\ell} \weakstarto v$ in $\BVV(\Omega)$, and hence $v^{n_\ell} \to v$ in $L^1(\Omega)$,
		there exists $\ell_2 \ge \ell_1$ such that for all $\ell \ge \ell_2$ the function $v$ is feasible for
		$\text{{\ref{eq:tr}}}(v^{n_\ell},\nabla F(v^{n_\ell}),\Delta^{n_\ell + 1,k_0})$.
		The feasibility of $v$ and \eqref{eq:pred_estimate_wrt_sigma_and_delta} yield
		$\pred(v^{n_\ell},\Delta^{n_\ell + 1,k_0}) \ge \alpha \delta - \frac{1-\sigma}{3-\sigma}\alpha\delta \ge \frac{2}{3}\alpha\delta> 2\frac{1-\sigma}{3-\sigma}\alpha\delta$.
		
		Thus if the inner loop reaches iteration $k_0$ for $\ell \ge \ell_0$ we obtain
		\begin{align*}
		\frac{\ared(v^{n_\ell}, \tilde{v}^{n_\ell + 1,k_0})}{\pred(v^{n_\ell},\Delta^{n_\ell + 1,k_0})}
		\ge \frac{\pred(v^{n_\ell},\Delta^{n_\ell + 1,k_0}) - 2 \frac{1-\sigma}{3-\sigma}\alpha\delta}{\pred(v^{n_\ell},\Delta^{n_\ell + 1,k_0})}
		\ge \frac{1 - 3 \frac{1-\sigma}{3-\sigma}}{1 - \frac{1-\sigma}{3-\sigma}}
		= \sigma,
		\end{align*}
		where the first inequality is due to \eqref{eq:pred_estimate_wrt_sigma_and_delta}
		and the second follows from the estimates above
		and the fact that
		$p \mapsto \frac{p - c}{p}$ is monotone ($c = 2\tfrac{1-\sigma}{3-\sigma}$).
		Consequently,
		\[ \ared(v^{n_\ell}, \tilde{v}^{n_\ell + 1,k_0}) \ge \sigma \pred(v^{n_\ell},\Delta^{n_\ell + 1,k_0}) \]
		and the iterate is accepted latemost in iteration $k_0$.
		Because the predicted reduction decreases with shrinking trust-region radii,
		the actual reduction in iteration $\ell$ is greater or equal than
		$\sigma \pred(v^{n_\ell},\Delta^{n_\ell + 1,k_0}) \ge \frac{2}{3} \sigma \alpha \delta$.
		Because the sequence of objective function values for the accepted iterates is
		monotonically decreasing, we obtain $J(v^{n_\ell + 1}) \to -\infty$, which is a contradiction.
		We conclude that $v^{n_\ell} \to v$ strictly in $\BVV(\Omega)$.
		
		\textbf{Outcome \ref{itm:infinite_seq_sl} (3) strict accumulation points
			are optimal for \eqref{eq:tr} if the trust-region radius is
			bounded away from zero:} Next, we assume that $\bar{v}$ with
		$\nabla F(\bar{v}) \in C(\bar{\Omega})$ is a weak-$^*$ and strict
		limit of a subsequence $(v^{n_\ell})_\ell$. Moreover, we assume
		that the trust-region radius upon acceptance of the iterates
		$v^{n_\ell+1}$ is bounded away from zero, that is
		$0 < \underline{\Delta} \coloneqq \inf_{\ell \in \N} \min_k \Delta^{n_\ell + 1,k}$.
		Because $0< \underline{\Delta}$ and $\Delta^{n_\ell + 1,k} = \Delta^0 2^{-k}$ for all
		inner iterations $k$, we may restrict to an infinite subsequence (for ease of notation
		denoted by the same symbol) such that the iterate $n_\ell + 1$ is accepted in
		iteration $k_0$ with $\underline{\Delta} =  \Delta^0 2^{-k_0}$.
		The $\Gamma$-convergence established in \cref{thm:tr_gamma_convergence} gives
		that every cluster point of $(v^{n_\ell + 1})_\ell$ minimizes
		$\text{\ref{eq:tr}}(\bar{v},\nabla F(\bar{v}),\underline{\Delta})$.
		Moreover, the optimal objective function values of the optimization problems 
		$\text{\ref{eq:tr}}(v^{n_\ell},\nabla F(v^{n_\ell}),\underline{\Delta})$,
		the predicted reductions upon acceptance, converge to zero because otherwise we would
		obtain the contradiction $J(v^{n_\ell + 1}) \to -\infty$. Thus the minimal objective of
		$\text{\ref{eq:tr}}(\bar{v},\nabla F(\bar{v}),\underline{\Delta})$ is zero, implying that
		$\bar{v}$ is optimal for $\text{\ref{eq:tr}}(\bar{v},\nabla F(\bar{v}),\underline{\Delta})$.
		In particular, $\bar{v}$ is L-stationary by virtue of \cref{prp:TRset_stationarity}.
		
		\textbf{Outcome \ref{itm:infinite_seq_sl} (4) strict accumulation points
			are L-stationary if the trust-region radius vanishes:} We close the proof by proving that
		if $v^{n_\ell} \weakstarto v$ and the trust-region radii upon acceptance of
		the iterates $v^{n_\ell + 1}$ vanish, then $v$ is L-stationary.
		We employ a contrapositive argument and assume that
		$v$ is not L-stationary. We have to show that the trust-region radius
		upon acceptance of the iterates $v^{n_\ell + 1}$ is bounded
		away from zero. Let $\ell \in \N$, let $\Delta^* \in \{ \Delta^0 2^{-j}\,|\, j \in \N\}$, and
		let $v^*$ minimize $\text{\ref{eq:tr}}(v^{n_\ell},\nabla F(v^{n_\ell}),\Delta^*)$. Then
		\begin{gather}\label{eq:aredpred_vnlvstar}
		\ared(v^{n_\ell}, v^*)
		\ge \sigma \pred(v^{n_\ell}, \Delta^*)
		+ (1 - \sigma)\pred(v^{n_\ell}, \Delta^*)
		- \frac{C}{2}(\Delta^*)^2
		\end{gather}
		by virtue of Taylor's theorem, the estimate from \cref{ass:hessian_regularity}
		and feasibility
		of $v^*$. Thus it is sufficient to show that $(1 - \sigma)\pred(v^{n_\ell}, \Delta^*) - \frac{C}{2}(\Delta^*)^2 \ge 0$
		holds for some $\Delta^*$ and all large enough $\ell$.
		
		Because $v$ is not L-stationary, there exist $\phi \in C_c^\infty(\Omega;\R^d)$ and $\eta > 0$ such that
		\[
		\begin{multlined}
		\sum_{i=1}^M \nu_i\int_{\partial^*{E}_i \cap \Omega} (-\nabla F(v))(x)(\phi(x)\cdot n_{{E}_i}(x))\dd \Ha^{d-1}(x)\\
		- \alpha \sum_{i=1}^M \sum_{j=i + 1}^M|\nu_i - \nu_j|
		\int_{\partial^*E_i \cap \partial^* E_j}\bdvg{E_i} \phi(x)\dd \Ha^{d-1}(x)
		> \eta,
		\end{multlined}
		\]
		where $\{E_1,\ldots,E_M\}$ is a Caccioppoli partition of $\Omega$ such that
		$v = \sum_{i=1}^M \nu_i\chi_{E_i}$.
		Let $(f_t)_{t\in(-\varepsilon,\varepsilon)}$ be the local variation
		defined by $f_t \coloneqq I + t \phi$. We obtain that
		\begin{gather}\label{eq:ftvv_decrease}
		-(\nabla F(v), f_t^{\#}v - v)_{L^2} - \alpha \TV(f_t^{\#}v) + \alpha \TV(v)
		\ge t \eta + g(t),
		\end{gather}
		where $g : [-\varepsilon,\varepsilon] \to \R$ is a function
		such that $g(t) \in o(t)$ by virtue of \cref{lem:Taylor_TV_mg2,lem:Taylor_lin_mg2}.
		\Cref{lem:tr_feasibility} implies that there exist $\kappa > 0$ and $\varepsilon_1 \in (0,\varepsilon)$
		such that
		$\|f_t^{\#}v - v\|_{L^1} \le |t|\kappa$ holds for all
		$t \in (-\varepsilon_1,\varepsilon_1)$.
		
		We choose $\Delta^*$ small enough ($j$ large enough) such that
		\begin{enumerate}[label=(\alph*)]
			\item\label{itm:Deltastar_st_ftvfeasible_a} $\Delta^* \le 2 \varepsilon_1 \kappa$, and
			\item\label{itm:Deltastar_st_remainder_dominance} $(1 - \sigma)\left(0.5 \eta\Delta^* \kappa^{-1}
			+ g\left(0.5 \Delta^* \kappa^{-1}\right)\right)
			- \left(2(1 - \sigma) + 0.5 C\right)(\Delta^*)^2 \ge 0$
		\end{enumerate}
		hold true. The second inequality can be satisfied because $g(t) \in o(t)$ and $\kappa > 0$ is constant. Then
		we choose $\ell_0 \in \N$ large enough such that for all $\ell \ge \ell_0$ we obtain
		\begin{enumerate}[label=(\alph*),resume]
			\item\label{itm:Deltastar_st_ftvfeasible_b} $\|v^{n_\ell} - v\|_{L^1} \le \frac{\Delta^*}{2}$.
		\end{enumerate}
		Let $t \coloneqq \tfrac{\Delta^*}{2 \kappa}$. 
		Then \ref{itm:Deltastar_st_ftvfeasible_a} gives $t\le \varepsilon_1$ and
		$\|f_t^{\#}v - v\|_{L^1} \le |t|\kappa
		\le \Delta^*/2$ gives
		\[ \|f_t{^\#}v - v^{n_\ell}\|_{L^1} \le \|f_t{^\#}v - v\|_{L^1} + \|v^{n_\ell} - v\|_{L^1} 
		\le \Delta^*,\]
		which implies that $f_t^{\#}v$ is feasible for
		$\text{\ref{eq:tr}}(v^{n_\ell},\nabla F(v^{n_\ell}),\Delta^*)$.
		
		The strict convergence of $v^{n_\ell}$ yields that there is
		$\ell_1 \ge \ell_0$ such that for all $\ell \ge \ell_1$
		\begin{gather}\label{eq:continuity_of_tr_obj_terms}
		\begin{aligned}
		|(\nabla F(v),f_t^{\#}v - v)_{L^2} - (\nabla F(v^{n_\ell}),f_t^{\#}v - v^{n_\ell})_{L^2}|\le (\Delta^{*})^2 &\text{ and}\\
		|\alpha \TV(v) - \alpha\TV(v^{n_\ell})| \le (\Delta^{*})^2 &.
		\end{aligned}
		\end{gather}
		Then we can estimate
		\begin{align*}
		\hspace{.5em}&\hspace{-.5em}
		(1 - \sigma)\pred(v^{n_\ell}, \Delta^*) - 0.5 C(\Delta^*)^2\\
		&\ge -(1 - \sigma)\left((\nabla F(v^{n_\ell}), f_t^{\#}v - v^{n_\ell})_{L^2} + \alpha \TV(f_t^{\#}v) - \alpha \TV(v^{n_\ell})\right)
		- 0.5 C(\Delta^*)^2\\
		&\ge -(1 - \sigma)\left((\nabla F(v), f_t^{\#}v - v)_{L^2} + \alpha \TV(f_t^{\#}v) - \alpha \TV(v)\right)
		- C_1(\Delta^*)^2 \\
		&\ge (1 - \sigma)\left(0.5 \Delta^* \eta \kappa^{-1} + g\left(0.5\Delta^*\kappa^{-1}\right)\right) - \left(2(1 - \sigma) + 0.5 C\right)(\Delta^*)^2,
		\end{align*}
		where the first inequality follows from the feasibility of $f_t^{\#}v$ for
		$\text{\ref{eq:tr}}(v^{n_\ell},\nabla F(v^{n_\ell}),\Delta^*)$ and the second and
		third one follow from \eqref{eq:continuity_of_tr_obj_terms} and \eqref{eq:ftvv_decrease}
		with the choice $C_1 \coloneqq 2(1 - \sigma) + C/2$.
		
		Because $\Delta^*$ has been chosen small enough such that
		\ref{itm:Deltastar_st_remainder_dominance} holds we
		have shown $(1 - \sigma)\pred(v^{n_\ell}, \Delta^*) - 0.5 C(\Delta^*)^2 \ge 0$ for all large enough $\ell$.
	\end{proof}
	
	\section{Computational Experiments}\label{sec:comp} In order to provide a first qualitative 
	assessment of the method in practice, we consider $\Omega = (0,1)^2$ and
	the following instance of \eqref{eq:p} that is governed
	by a stationary advection-diffusion equation
	with a homogeneous Dirichlet boundary condition on three sides,
	$\Gamma_{D} = [0,1] \times \{0, 1\} \cup \{0\} \times (0,1)$,
	and a free boundary condition on the remaining side
	%
	\[ 
	\min_{y,w} \frac{1}{2}\|y - y_d\|_{L^2(\Omega)}^2
	+ \alpha \TV(w) \text{ s.t.\ }
	\left\{
	\begin{aligned}
	-\varepsilon \Delta y + b \cdot \nabla y &= w, \\
	w|_{\Gamma_{D}} &= 0, \\
	w(x) &\in \{0,1,2\}.
	\end{aligned}
	\right.
	\] 
	Therein, we choose the constants $\varepsilon = 1.5 \cdot 10^{-2}$,
	$b = \begin{pmatrix} \cos(\pi / 32) & \sin(\pi / 32)\end{pmatrix}^T$,
	and $\alpha = 10^{-4}$.
	In order to solve the PDE, we discretize its variational formulation with 
	the finite-element package \textsc{Firedrake} \cite{rathgeber2016firedrake}.
	We consider a fixed discretization of $\Omega$ into a grid of $64 \times 64$
	squares, which are decomposed into four triangles. We use continuous
	Lagrange finite elements of order one on the triangular grid for the state 
	variable and piecewise constant functions on the square grid for the control
	variable. 
	
	As pointed out in \cite{leyffer2021sequential}, the trust-region subproblems
	\eqref{eq:tr} become linear integer programs after discretization,
	which is briefly summarized in \Cref{sec:subproblems_as_ips}.
	We solve integer programming formulations of the discretized subproblems
	with the general purpose integer programming solver 
	\textsc{Gurobi}  \cite{gurobi}. For the trust-region algorithm  we
	choose $\Delta^0 = 0.125$  and $\sigma = 10^{-4}$. We execute our 
	computational experiments on a laptop computer with an Intel(R)
	Core(TM) i9-10885H CPU (2.40\;GHz) and 32\;GB RAM. 
	
	The execution of \cref{alg:trm} takes 73 iterations and
	takes 100 minutes. \Cref{alg:trm} terminates when the trust-region
	contracts to zero (that is $\Delta$ falls below $\lambda(\Omega) / 64^2$).
	The final objective value that is reached is $5.69 \cdot 10^{-4}$,
	where the tracking-type term has the value $5.64 \cdot 10^{-5}$ and
	the $\alpha \TV$-term has the value $5.13 \cdot 10^{-4}$.
	
	For comparison, a continuous relaxation with the choices $\alpha = 0$
	and $w(x) \in [0,2]$ initialized with $w^0 \equiv 0$ with the control
	optimized on the triangular discretization takes two minutes to solve
	using the solver \textsc{PETSc TAO} \cite{tao}. The final objective
	value (only consisting of the tracking-type term) that is reached
	is $2.24 \cdot 10^{-5}$.
	
	Because a numerical analysis of the algorithm and experiments
	on test problem libraries are beyond of the scope of this work, we
	refrain from speculating about convergence speed, etc.\ but still like to
	show the convergence behavior of our implementation of
	\cref{alg:trm} for our example. To this end, we plot the trust-region radius on acceptance and the objective value
	over the iterations in \cref{fig:tr_ov_over_iters}.
	\begin{figure}
		\centering
		\begin{subfigure}{.5\textwidth}
			\centering
			\includegraphics[width=.95\linewidth]{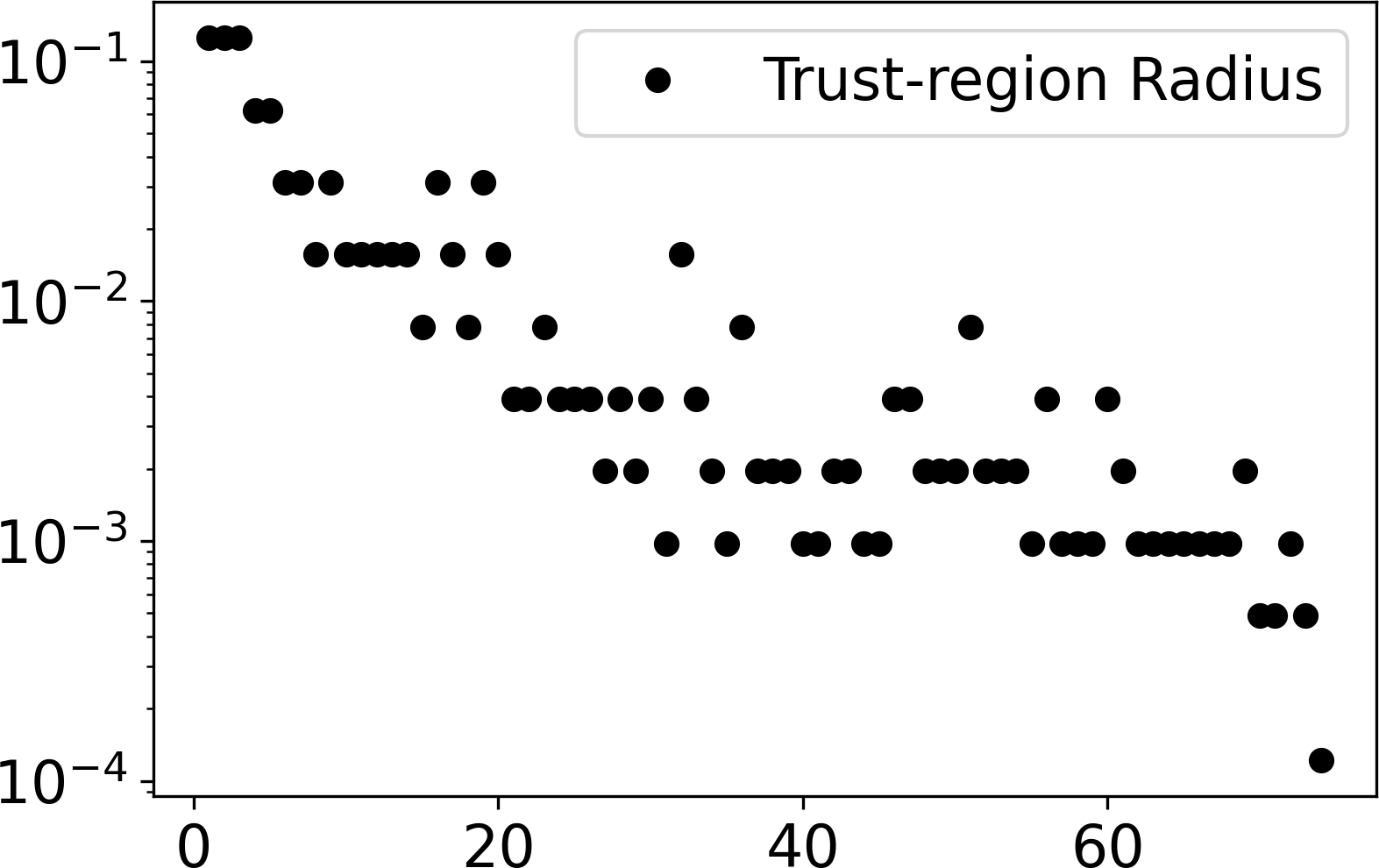}
		\end{subfigure}%
		\begin{subfigure}{.5\textwidth}
			\centering
			\includegraphics[width=.95\linewidth]{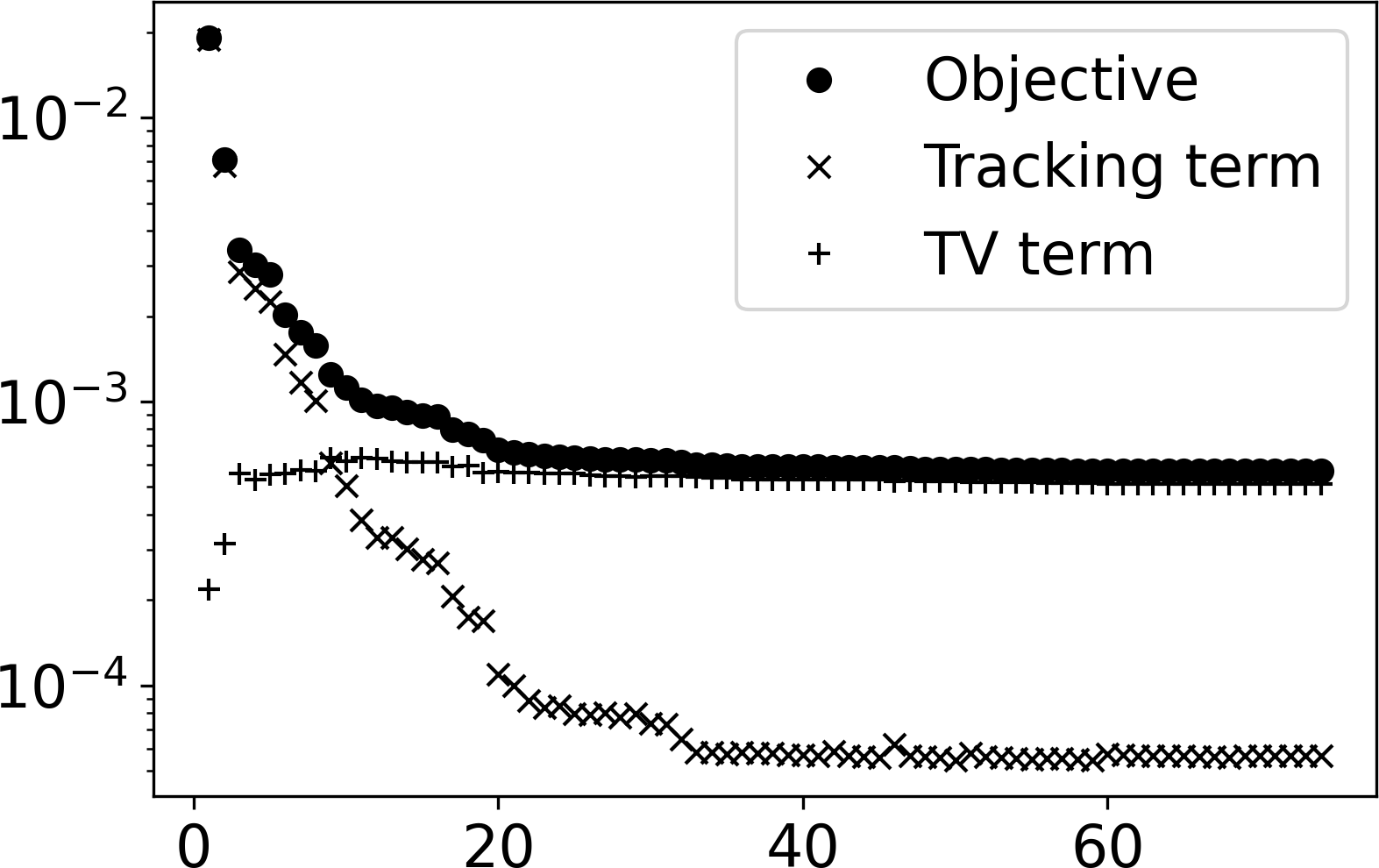}
		\end{subfigure}
		\caption{Trust-region radius upon acceptance (left) and
			objective value (right) achieved over the iterations of
			\cref{alg:trm}.}
		\label{fig:tr_ov_over_iters}
	\end{figure}
	
	Our implementation of \cref{alg:trm} is able 
	to compute a non-trivial control whose level sets partition the 
	control domain. In particular, the level sets to the values $1$ and
	$2$ of the final control consist of several disjoint connected
	components. We illustrate this by plotting
	the controls and corresponding states for iterations 3, 25, and
	73 in \cref{fig:ctrl_state_over_iterations}.
	
	\begin{figure}
		\centering
		\begin{subfigure}{.25\textwidth}
			\centering
			\includegraphics[width=\linewidth]{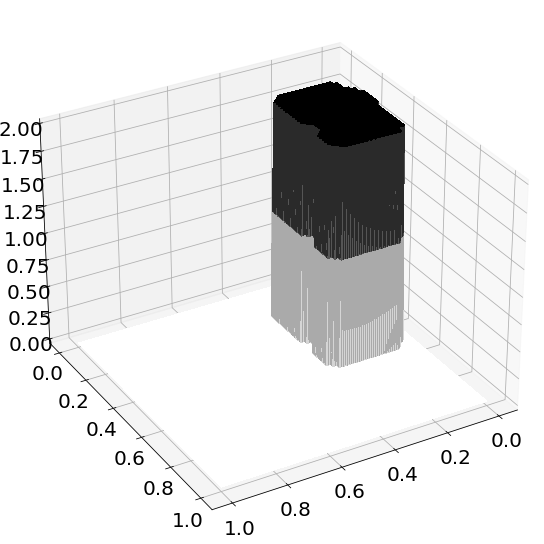}
		\end{subfigure}%
		\begin{subfigure}{.25\textwidth}
			\centering
			\includegraphics[width=\linewidth]{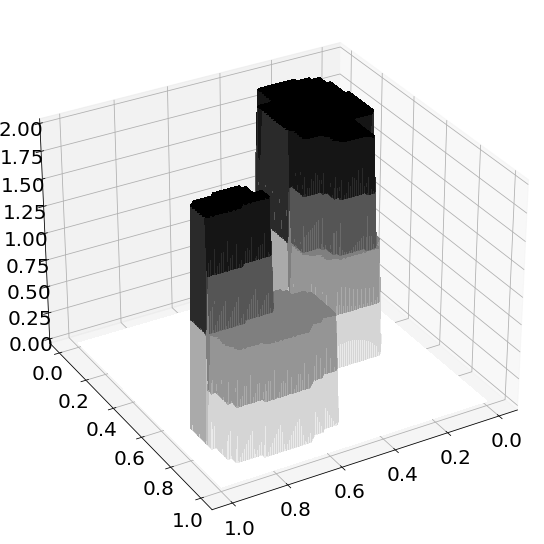}
		\end{subfigure}%
		\begin{subfigure}{.25\textwidth}
			\centering
			\includegraphics[width=\linewidth]{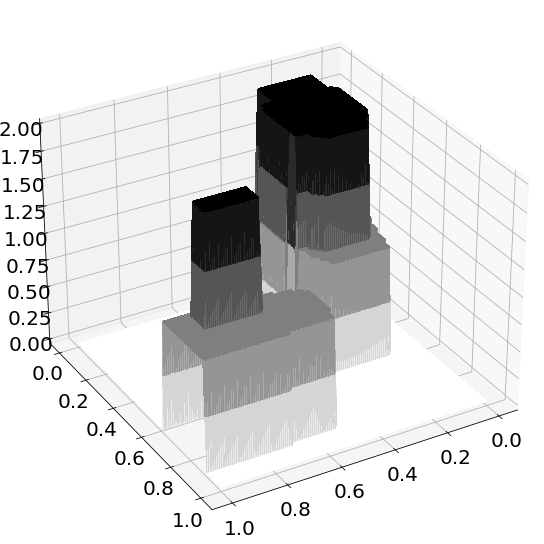}
		\end{subfigure}%
		\begin{subfigure}{.25\textwidth}
			\centering
			\includegraphics[width=\linewidth]{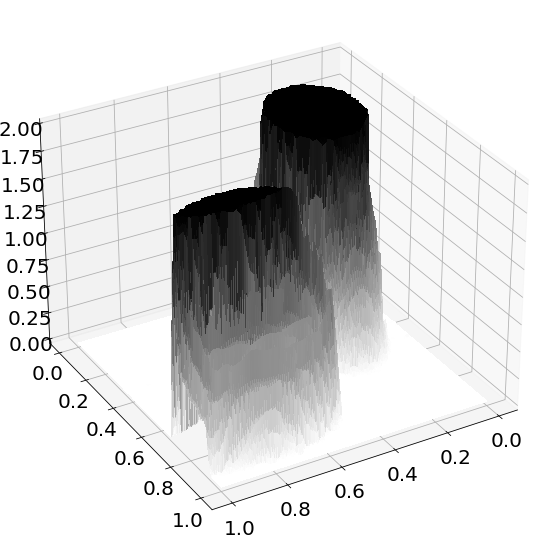}
		\end{subfigure}\\
		\begin{subfigure}{.25\textwidth}
			\centering
			\includegraphics[width=\linewidth]{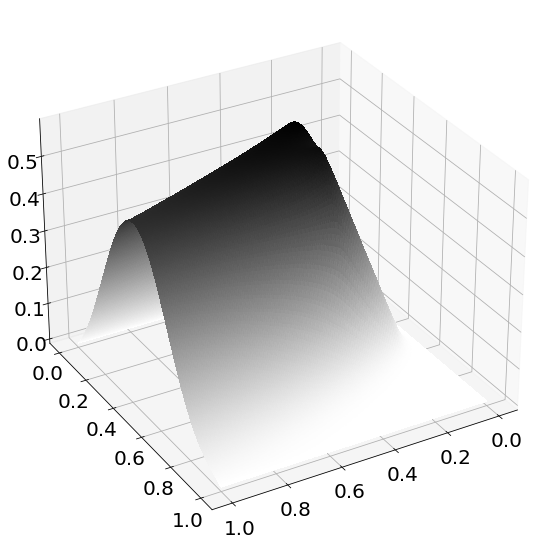}
			\caption{It \# 3}
		\end{subfigure}%
		\begin{subfigure}{.25\textwidth}
			\centering
			\includegraphics[width=\linewidth]{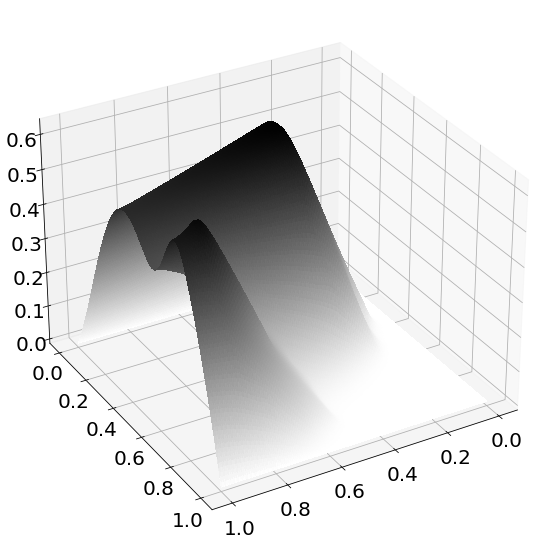}
			\caption{It \# 25}
		\end{subfigure}%
		\begin{subfigure}{.25\textwidth}
			\centering
			\includegraphics[width=\linewidth]{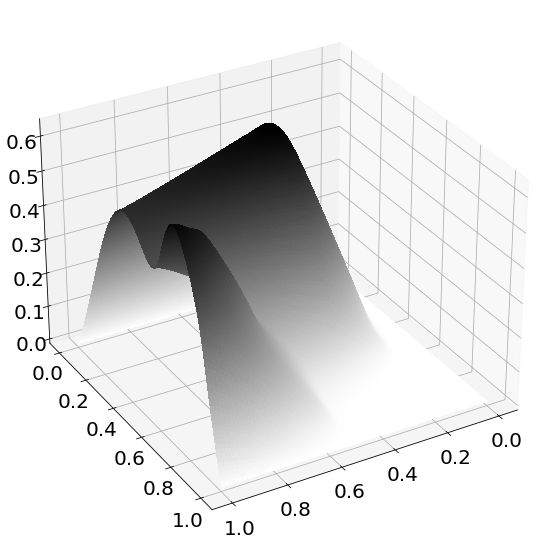}
			\caption{It \# 73}
		\end{subfigure}%
		\begin{subfigure}{.25\textwidth}
			\centering
			\includegraphics[width=\linewidth]{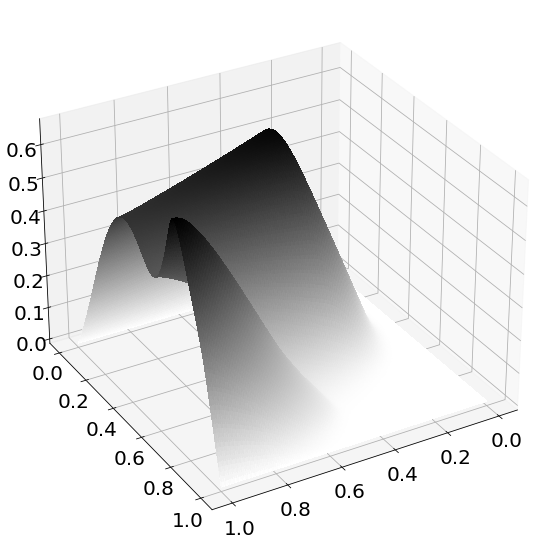}
			\caption{Relaxation}
		\end{subfigure}
		\caption{Control iterates (top) and corresponding
			states (bottom) produced by \cref{alg:trm}
			for iterations 3, 25, and 73 as well as for the
			final iteration of the relaxation.}
		\label{fig:ctrl_state_over_iterations}
	\end{figure}
	
	\section{Conclusion}\label{sec:conclusion}
	
	We have investigated TV-regularized integer optimal control problems on multi-dimensional domains. We have derived a first-order optimality condition for locally optimal solutions of \cref{eq:p} by means of local variations that yield feasible perturbations of the level sets of our integer-valued control inputs. In order to solve \cref{eq:p}, we have analyzed the asymptotics of the trust-region algorithm proposed in \cite{leyffer2021sequential} for the case $d \geq 2$. We have proved $\Gamma$-convergence of the trust-region subproblems with respect to strict convergence of the linearization point. We have established a sufficient decrease condition based on local variations, which in turn leads to convergence of the iterates of the trust-region algorithm to first-order optimal points.
	
	In order to obtain a useful method in practice, it is important to extend 
	our function space analysis with sophisticated discretizations
	and a corresponding numerical analysis. Moreover, we have experienced
	numerical difficulties and very long compute times of the
	integer programming solver when choosing finer 
	control discretizations for the example in \cref{sec:comp}
	so that scaling and stabilization techniques as well as efficient 
	algorithms for solving the discretized trust-region subproblems 
	remain open questions as well.
	
	\section*{Acknowledgments}\label{sec:acknowledgments}
	The authors thank two anonymous referees for providing helpful feedback on the manuscript.
	
	\appendix
	
	\section{Auxiliary Results} \label{sec:auxilary}
	\begin{proof}[Proof of \cref{lem:TV_of_BVV}]
		Claim (a) follows, e.g., from \cite[Lem.\ 2.2]{leyffer2021sequential}. Claim (b) follows from \cite[Thm.\  3.40]{ambrosio2000functions}, which states that for $v \in \BVV(\Omega)$ the sets
			\begin{align*}
			F_i \coloneqq \left\{ x \in \Omega \, \middle| \, v(x) > \nu_i - \varepsilon \right\} = \left\{ x \in \Omega \, \middle| \, v(x) \geq \nu_i \right\}
			\end{align*}
			with $\varepsilon \in (0,1)$ and $i \in \{1,\dots,M\}$ are of finite perimeter in $\Omega$. Without loss of generality, we may assume $\nu_i > \nu_{i+1}$ for all $i \in \{1, \dots, M-1\}$. We define $E_1 \coloneqq F_1 = v^{-1}(\{\nu_1\})$ and $E_i \coloneqq F_i \setminus F_{i-1} = v^{-1}(\{\nu_i\})$ for all $i \in \{2,\dots,M\}$. This yields $P(E_1,\Omega) = P(F_1,\Omega) < \infty$, $P(E_i, \Omega) \leq P(F_i, \Omega) + P(F_{i-1},\Omega) < \infty$ for all $i \in \{ 2,\dots,M \}$, and $v = \sum_{i=1}^{M} \nu_i \chi_{E_i}$.
			
			In order to prove (c), let $\sum_{i=1}^{M} \nu_i \chi_{E_i} = v \in \BVV(\Omega)$ as in (b). By means of Theorems 3.36 and 3.59 in
			\cite{ambrosio2000functions}, we obtain
			\begin{gather}\label{eq:first_char}
			Dv = \sum_{i=1}^M \nu_i D\chi_{E_i}
			= -\sum_{i=1}^M \nu_i n_{E_i} \Ha^{d-1} \mres (\partial^* E_i \cap \Omega),
			\end{gather}
			where the functions $\chi_{E_i}$ are considered as elements of $\BV(\Omega)$.
			Next, we prove the identity
			\begin{gather}\label{eq:dv_identity}
			Dv = \sum_{i=1}^{M-1} \sum_{j=i+1}^{M} (\nu_j-\nu_i) n_{E_i} \mathcal{H}^{d-1} \mres (\partial^* E_i \cap \partial^* E_j \cap \Omega),
			\end{gather}
			where $Dv$ denotes the Radon measure that is
			the distributional derivative of $v \in \BVV(\Omega)$
			and $n_{E_i}$ denotes the unit outer normal vector of $E_i$
			that is defined on $\partial^* E_i$.
			To this end, we observe that every $x \in \partial^* E_i \cap \Omega$ is a point of density $\tfrac 12$ for $E_i$
			and, consequently, cannot be a point of density $1$ for any $E_j$, $j \in \{1,\ldots,M\}$. We apply
			Theorem 4.17 in \cite{ambrosio2000functions} to the right hand side of \eqref{eq:first_char} and obtain
			\[ 	Dv = -\sum_{i=1}^M \nu_i n_{E_i} \Ha^{d-1} \mres \Bigg(\underset{j \neq i}{\bigcup_{j = 1}^M} \underbrace{\partial^* E_i \cap \partial^* E_j \cap \Omega}_{\eqqcolon A_{ij}}
			\Bigg),
			\]
			where we have used that $x \in \Omega$ can be a point of density $\tfrac 12$ for at most two of the disjoint sets $E_k$, $k \in \{1,\ldots,M\}$.
			Using this observation again, we obtain that the sets $A_{ij}$ are pairwise disjoint, which implies
			\begin{gather}\label{eq:intermediate_claim}
			Dv = -\sum_{i=1}^{M-1}\sum_{j = i + 1}^M (\nu_i n_{E_i} + \nu_{E_j} n_{E_j}) \Ha^{d-1} \mres (\partial^* E_i \cap \partial^* E_j \cap \Omega).
			\end{gather}
			Then \eqref{eq:dv_identity} follows because $n_{E_j} = - n_{E_i}$ on $\partial^* E_i \cap \partial^* E_j$.
			Because every $x \in \partial^* E_i \cap \partial^* E_j$ has density $\frac{1}{2}$ with respect to $E_i$ and $E_j$ and therefore density $1$
			with respect to $\Omega$, which yields $x \notin \partial^* \Omega$, which is a subset of
			the points of density $\tfrac{1}{2}$
			with respect to $\Omega$.
		Since $\Omega$ has Lipschitz boundary, there holds $\Ha^{d-1}(\partial \Omega \setminus \partial^* \Omega) = 0$, see \cite[Prop.\ 3.62]{ambrosio2000functions}. In combination, we get
			\begin{gather*}
			\Ha^{d-1}((\partial^* E_i \cap \partial^* E_j) \setminus \Omega) = 0.
			\end{gather*}
			Moreover, because the $A_{ij}$ are pairwise disjoint, the measures $\Ha^{d-1} \mres A_{ij}$ are pairwise singular.
			Thus we deduce with $\TV(v) = |Dv|(\Omega)$ and $\|n_{E_i}(x)\| = 1$ for 
			\eqref{eq:dv_identity} that		
			\begin{gather*}
			\TV(v) 
			= \sum_{i=1}^{M-1} \sum_{j = i+1}^{M} | \nu_i - \nu_j | \mathcal{H}^{d-1}(\partial^* E_i \cap \partial^* E_j),
			\end{gather*}
			which gives \eqref{eq:tv_identity}. Moreover,
			\begin{gather*}
			\begin{aligned}
			\infty > \TV(v)
			&= \sum_{i=1}^{M - 1} \sum_{j = i+1}^{M} | \nu_i - \nu_j | \mathcal{H}^{d-1}(\partial^* E_i \cap \partial^* E_j) \\
			&\ge \frac{1}{2}\sum_{i=1}^{M}\underset{j \neq i}{\sum_{j = 1}^{M}} \mathcal{H}^{d-1}(\partial^* E_i \cap \partial^* E_j) \\
			&\ge \frac{1}{2} \sum_{i=1}^M\mathcal{H}^{d-1}(\partial^* E_i \cap \Omega)
			= \frac{1}{2} \sum_{i=1}^M P(E_i,\Omega),
			\end{aligned}
			\end{gather*}
			which gives \eqref{eq:Ek_has_finite_perimeter}.
	\end{proof}
	\begin{lemma}\label{lem:derivative_arg}
		Let $g_i : \R^d \to \R^d$, $i \in J$ for a (potentially uncountable) index set $J$ be a family of functions such that
			for $\delta \in (0,1)$ there exists a set $J_\delta \subset J$ such that for all $x \in \R^d$ it holds that $\|\nabla g_{i}(x) - I\| \le \delta$
			for all $i \in J_\delta$. Then for all $i \in J_{\delta}$ the function $g_i$ is a diffeomorphism.
	\end{lemma}
	\begin{proof}
		The uniform estimate $\|\nabla g_i(x) - I\| \le \delta$ implies that for $\delta < 1$
		the functions $G_i^y : \R^d \ni x \mapsto y + (x - g_i(x)) \in \R^d$, $i \in J_\delta$, have a unique fixed point
		for all $y \in \R^d$ by virtue of the Banach fixed point theorem. In particular, $x = G^y_i(x)$ if and only if $y = g_i(x)$.
		Thus $g_i$ is invertible. Moreover, $\nabla g_i(x)$ is invertible
		for all $x\in \R^d$
		and thus the inverse function theorem implies that $g_i$ is a diffeomorphism.
	\end{proof}
	\begin{lemma}\label{lem:g_t}
		Let $f_t \coloneqq I + t \phi$ for $\phi \in C_c^\infty(\Omega;\R^d)$. Then there is $\varepsilon > 0$ such that $g_t \coloneqq f_t^{-1}  = I - t \phi \circ g_t$ for $t \in (- \varepsilon, \varepsilon)$. Moreover, the mapping $(-\varepsilon , \varepsilon) \ni t \mapsto g_t(y) \in \R^d$ is Lipschitz continuous for each $y \in \R^d$ and $\nabla g_t(y) \to I$ as $t \to 0$ uniformly for $y \in \R^d$.
	\end{lemma}
	\begin{proof}
		Let $\phi \in C_c^\infty(\Omega;\R^d)$. We denote the Lipschitz constant of $\phi$ by $L_\phi$. Let $\varepsilon < \frac{1}{L_\phi}$. Define $T_y : \R^d \to \R^d$, $T_y \coloneqq y - t \phi$, for fixed $t \in (-\varepsilon,\varepsilon)$ and $y \in \R^d$. Then $T_y$ is a contraction mapping since for arbitrary $x,z \in \R^d$ it holds that
		\[
		\| T_y(x) - T_y(z) \| = | t | \| \phi(x)-\phi(z) \| \leq | t | L_\phi \| x - z \|.
		\]
		By Banach's fixed point theorem, there exists a unique fixed point $\tilde{x}_y \in \R^d$ with $T_y(\tilde{x}_y) = \tilde{x}_y$. We define $g_t(y) \coloneqq \tilde{x}_y$ and prove that $g_t = f_t^{-1}$. It holds that $g_t(y) = T_y(\tilde{x}_y) = y - t \phi(\tilde{x}_y) = y - t \phi(g_t(y))$ for $y \in \R^d$ and therefore
		\begin{align*}
		\| g_t(f_t(y)) - y \| &= 
		\| t \phi(y) - t \phi(g_t(f_t(y))) \|
		\leq |t| L_\phi \|g_t(f_t(y)) - y \|.
		\end{align*}
		Since $|t|L_\phi < 1$, it must hold that $\| g_t(f_t(y)) - y \|  = 0$.
		
		We now prove the Lipschitz continuity of $g_t(y)$ in $t \in (-\varepsilon,\varepsilon)$ for fixed $y \in \R^d$. For this, let $t,s \in (-\varepsilon,\varepsilon)$. Then
		\begin{align*}
		\| g_t(y) - g_s(y) \| & = \| s \phi(g_s(x)) - t \phi(g_t(y)) \| 
		\leq |s-t| \| \phi(g_s(y)) \| + |t| L_\phi \|g_s(y) - g_t(y) \|
		\end{align*}
		and therefore
		\begin{align*}
		\| g_t(y) - g_s(y) \| \leq |t-s| \frac{\| \phi(g_s(y)) \|}{1- |t| L_\phi} < |t-s| \frac{\| \phi(g_s(y)) \|}{1- \varepsilon L_\phi} \leq |t - s | C
		\end{align*}
		with $C = (1 - \varepsilon L_\phi)^{-1} \max_{x \in \bar{\Omega}} \| \phi(x) \| < \infty$, only depending on $\phi$ but not on $y$.
		
		In order to prove the convergence of $\nabla g_t(y)$, we now choose $\varepsilon$ small enough so that \cref{lem:derivative_arg} gives the existence of $(\nabla f_t(x))^{-1}$, which we write as its Neumann series $(\nabla f_t(x))^{-1} = I + \sum_{k=1}^{\infty} (-1)^k t^k \nabla \phi(x)^k$. By the inverse function theorem we have
		\[
		\nabla g_t(y) = (\nabla f_t(g_t(y)))^{-1} = I + \sum_{k=1}^\infty (-1)^k t^k \nabla \phi(g_t(y))^k.
		\]
		Thus, for $|t|$ sufficiently small (that is by potentially further reducing $\varepsilon$)
		\begin{align*}
		\| I - \nabla g_t(y) \| &\le
		\sum_{k=1}^{\infty} |t|^k \| \nabla \phi (g_t(y))\|^k 
		\leq  \sum_{k=1}^{\infty} |t|^k M^k = \frac{|t| M}{1 -|t| M } \to 0
		\end{align*}
		as $|t| \to 0$
		uniformly for all $y \in \R^d$ with $M \coloneqq \max_{x \in \bar{\Omega}} \| \nabla \phi (x) \| < \infty$.
	\end{proof}
	
		\begin{lemma}\label{lem:boundary_union}
			Let $E,F \subset \Omega$ be sets of finite perimeter. Then there holds
			\begin{align*}
			\partial^*(E \cup F) \subset \partial^e E \cup \partial^e F \text{ and } \partial^*(E \setminus F) \subset \partial^e E \cup \partial^e F.
			\end{align*}
			In particular, there holds
			\begin{align*}
			\Ha^{d-1}(\partial^*(E \cup F)) \leq \Ha^{d-1}(\partial^* E \cup \partial^* F) \text{ and } \Ha^{d-1}(\partial^*(E \setminus F)) \leq \Ha^{d-1} (\partial^* E \cup \partial^* F).
			\end{align*}
		\end{lemma}
		\begin{proof}
			Let $x \in \partial^* (E \cup F)$. Then by \cite[Thm. 3.61]{ambrosio2000functions}, there holds that $x$ has density $\frac{1}{2}$ regarding $E \cup F$, and therefore
			\begin{align*}
			\frac{1}{2} = \lim_{r \searrow 0}  \frac{\lambda((E \cup F) \cap B_r(x) ) )}{\lambda(B_r(x))} \geq \limsup_{r \searrow 0} \frac{\lambda( E \cap B_r(x) ) }{\lambda(B_r(x))}
			\end{align*}
			and
			\begin{align*}
			\frac{1}{2} = \lim_{r \searrow 0}  \frac{\lambda((E \cup F) \cap B_r(x) ) )}{\lambda(B_r(x))} \geq \limsup_{r \searrow 0} \frac{\lambda( F \cap B_r(x) ) }{\lambda(B_r(x))}.
			\end{align*}
			Moreover,
			\begin{align*}
			\frac{1}{2} = \lim_{r \searrow 0}  \frac{\lambda((E \cup F) \cap B_r(x) ) )}{\lambda(B_r(x))} \leq \liminf_{r \searrow 0} \frac{\lambda( E \cap B_r(x) ) }{\lambda(B_r(x))} + \frac{\lambda( F \cap B_r(x) ) }{\lambda(B_r(x))}
			\end{align*}
			and thus
			\begin{align*}
			\frac{\lambda( E \cap B_r(x) ) }{\lambda(B_r(x))} \not \to 0 \quad \text{ or } \quad \frac{\lambda( F \cap B_r(x) ) }{\lambda(B_r(x))} \not \to 0 \quad \text{ as } r \searrow 0.
			\end{align*}
			Therefore, $x \in \partial^e E \cup \partial^e F$. To prove $\partial^* (E \setminus F) \subset \partial^e E \cup \partial^e F$, we use that
			\begin{align*}
			\R^d \setminus (E \setminus F) = (\R^d \setminus E) \cup F
			\end{align*}
			and obtain
			\begin{align*}
			\partial^*( E \setminus F ) = \partial^*(\R^d \setminus (E \setminus F)) = \partial^* ( (\R^d \setminus E ) \cup F ) \subset \partial^e (\R^d \setminus E) \cup \partial^e F = \partial^e E \cup \partial^e F.
			\end{align*}
			The last claim follows  from the monotonicity of $\Ha^{d-1}$ and \cite[Thm. 3.61]{ambrosio2000functions}.
	\end{proof}
	
	\section{Discretized Trust-region Subproblems as Linear Integer Programs}\label{sec:subproblems_as_ips}
	
		We briefly summarize how the trust-region subproblems yield
		integer linear programs for a fixed control discretization, see also
		\cite[Section 3.3]{leyffer2021sequential}.
		To this end, we fix a partition $\mathcal{P}$ of $\Omega$ into finitely many
		polytopes of dimension $d$. We denote the set of interior facets by
		$\mathcal{E} \subset \mathcal{P}\times \mathcal{P}$.
		The piecewise constant functions on this partition that attain values in $V$
		can be written as $v(x) = \sum_{P\in\mathcal{P}} v_P \chi_{P}(x)$
		for a.a.\ $x \in \Omega$ with $v_P \in V$. Moreover, for
		$E = (P_a, P_b) \in \mathcal{E}$, we write
		$[v]_E = v_{P_a} - v_{P_b}$. For such functions $v$, $\bar{v}$
		and $g \in L^2(\Omega)$
		we can rewrite the terms in the trust-region subproblems \eqref{eq:tr}
		as:
		\begin{align*}
		\TV(v) &=
		\sum_{E \in \mathcal{E}} \int_E \left|[v]_E\right|\dd \Ha^{d-1}
		= \sum_{E \in\mathcal{E}} \Ha^{d-1}(E) \left|[v]_{E}\right|,\\
		(g, v - \bar{v})_{L^2(\Omega)}
		&= \sum_{P \in \mathcal{P}}(v_P - \bar{v}_P) \underbrace{\int_{P} g(x) \dd x}_{\eqqcolon c_P},\\
		\|v - \bar{v}\|_{L^1(\Omega)} &= \sum_{P\in\mathcal{P}} |v_P - \bar{v}_P| \lambda(P).
		\end{align*}
		By means of auxiliary variables, we can transform the absolute values into
		linear inequalities, which gives the resulting integer linear program formulation:
		\begin{gather*}
		(\text{TR-as-IP}(\bar{v},g,\Delta))
		\coloneqq
		\left\{
		\begin{aligned}
		\min_{v_P,u_P,w_E,\tau}\ &
		\sum_{P \in \mathcal{P}} c_P (v_P - \bar{v}_{P})
		+ \alpha \tau - \alpha \TV(\bar{v}) \\
		\text{~~s.t.~~~~}\ & -u_P \le v_P - \bar{v}_P\le u_P \text{ for all }
		P\in\mathcal{P}, \\
		& {\sum_{P \in \mathcal{P}}} u_P\lambda(P)
		\le \Delta, \\
		& -w_E \le v_{P_a} - v_{P_b} \le w_E
		\text{ for all } E = (P_a,P_b) \in \mathcal{E}, \\
		& {\sum_{E \in \mathcal{E}}} w_E \Ha^{d-1}(E) \le \omega, \\
		& v_P \in \{\nu_1,\ldots,\nu_M\}
		\text{ for all } P \in\mathcal{P}.
		\end{aligned}
		\right.
		\end{gather*}

	\bibliographystyle{plain}
	\bibliography{biblio}
	
\end{document}